%% file: arcq_ams.tex
\documentclass{amsart}   

\usepackage{amspaper}
\usepackage{lmodern}  
\crefname{subsection}{Section}{Sections}
\crefname{assumption}{Assumption}{Assumptions}

\usepackage{bm}
\usepackage{algoE}
\usepackage{algorithmicx,algpseudocode}

\newcommand{\arcq}{ARC\(_q\)\xspace}
\newcommand{\arcqk}{ARC\(_q\)K\xspace}

\makeatletter
\newcounter{assumptionbkup}
\newenvironment{assumptionb}[1]{%
  \setcounter{assumptionbkup}{\value{assumption}}
  \renewcommand{\theassumption}{\ref*{#1}b}%
  \renewcommand{\@currentlabel}{\protect\ref*{#1}b}
  \csname phantomsection\endcsname
  \begin{assumption}
}{%
  \end{assumption}
  \setcounter{assumption}{\value{assumptionbkup}}
}
\newenvironment{assumptionc}[1]{%
  \setcounter{assumptionbkup}{\value{assumption}}
  \renewcommand{\theassumption}{\ref*{#1}c}%
  \renewcommand{\@currentlabel}{\protect\ref*{#1}c}
  \csname phantomsection\endcsname
  \begin{assumption}
}{%
  \end{assumption}
  \setcounter{assumption}{\value{assumptionbkup}}
}
\makeatother

\renewcommand{\year}{2015}     
\newcommand{\cahiernumber}{109}  

\title[ARCqK]{%
  Scalable adaptive cubic regularization methods
}

\author[J.-P. Dussault]{Jean-Pierre Dussault}
\address{%
  GERAD and
  D\'epartement d'Informatique,
  Universit\'e de Sherbrooke,
  Sherbrooke, QC, Canada.
}
\urladdr{\http{www.dmi.usherb.ca/~dussault}}
\email{\mailto{Jean-Pierre.Dussault@USherbrooke.CA}}
\thanks{Research partially supported by an NSERC Discovery Grant}

\author[D. Orban]{Dominique Orban}
\address{%
  GERAD and
  Mathematics and Industrial Engineering Department \\
  Polytechnique Montr\'eal,
  Montr\'eal, QC, Canada.
}
\urladdr{\http{www.gerad.ca/~orban}}
\email{\mailto{dominique.orban@gerad.ca}}
\thanks{Research partially supported by an NSERC Discovery Grant}



\begin{document}
  \begin{abstract}
    Adaptive cubic regularization (ARC) methods for unconstrained optimization compute steps from linear systems involving a shifted Hessian in the spirit of the Levenberg-Marquardt and trust-region methods.
    The standard approach consists in performing an iterative search for the shift akin to solving the secular equation in trust-region methods.
    Such search requires computing the Cholesky factorization of a tentative shifted Hessian at each iteration, which limits the size of problems that can be reasonably considered.
    We propose a scalable implementation of ARC named \arcqk in which we solve a set of shifted systems concurrently by way of an appropriate modification of the Lanczos formulation of the conjugate gradient (CG) method.
    At each iteration of \arcqk to solve a problem with \(n\) variables, a range of \(m \ll n\) shift parameters is selected.
    The computational overhead in CG beyond the Lanczos process is thirteen scalar operations to update five vectors of length \(m\) and two \(n\)-vector updates for each value of the shift.
    The CG variant only requires one Hessian-vector product and one dot product per iteration, independently of the number of shift parameters.
    Solves corresponding to inadequate shift parameters are interrupted early.
    All shifted systems are solved inexactly.
    Such modest cost makes our implementation scalable and appropriate for large-scale problems.
    We provide a new analysis of the inexact ARC method including its worst case evaluation complexity, global and asymptotic convergence.
    We describe our implementation and provide preliminary numerical observations that confirm that for problems of size at least 100, our implementation of ARCqK is more efficient than a classic Steihaug-Toint trust region method.
    Finally, we generalize our convergence results to inexact Hessians and nonlinear least-squares problems.
  \end{abstract}

  \nolinenumbers
  \maketitle

  \thispagestyle{firstpage}
  \pagestyle{myheadings}

  \input{arcq_mainJPD}

  \small
  \bibliographystyle{abbrvnat}
  \bibliography{abbrv,arcqk,ScalableARC}
  \normalsize

  \newpage
  \hypertarget{contents}{}
  \tableofcontents

\end{document}


\maketitle

\section{Legend}

The table headers are as follows:

\begin{itemize}
  \item ``name'' is the problem name;
  \item ``nvar'' is the number of variables;
  \item $f(x)$ is the final objective value;
  \item $\|g(x)\|$ is the final gradient $\ell_2$-norm;
  \item ``iter'' is the number of iterations;
  \item $\# f$, $\# g$ and $\# Hv$ is the number of objective and gradient evaluations, and of Hessian-vector products, respectively;
  \item $t$ is the solve time;
  \item ``stat'' is the final status, which is blank in case of success, ``x'' if an exception occurs, ``e'' if the time budget is exceeded, or ``?'' if a different error occurs.
\end{itemize}

\section{ARCqK Results}

\input{arcqk_full_results.tex}

\section{SteihaugToint Results}

\input{st_full_results.tex}


%% file: arcq_mainJPD.tex

\section{Introduction}%
\label{sec:Intro}

We consider the unconstrained problem
\begin{equation}
  \label{eq:prob}
  \minimize{x\in \R^n} \ f(x)
\end{equation}
where \(f: \R^n \to \R\) is \(\mathcal{C}^2\) and \(\nabla^2 f\) is Lipschitz continuous.

Adaptive Cubic Regularization (ARC) algorithms, recently explored by~\citet{cartis-gould-toint-2011a,cartis-gould-toint-2011b} are closely related to trust region (TR) methods~\citep{conn-gould-toint-2000} in that steps are computed by solving a sequence of regularized subproblems.
A major theoretical appeal of ARC over TR methods is their optimal worst-case complexity property.
Whereas the number of function evaluations required to reach a point \(x\) for which \(\|\nabla f(x)\|\le\epsilon\) is \(O(\epsilon^{-2})\) for TR, which is no better than steepest descent~\citep{CARTIS201293}, that number is \(O(\epsilon^{-3/2})\) for ARC\@.
\citet{dussault-2015} develops the \arcq variant, which uses the usual quadratic model but computes regularized Newton steps satisfying the cubic subproblem's optimality conditions, and obtains simple proofs highlighting the key properties that ensure optimal worst-case complexity.

Both \arcq and TR algorithms employ the quadratic model
\begin{equation}
  \label{eq:quad-model}
  q_x(d) = f(x) + \nabla f(x)^T d + \tfrac{1}{2} d^T \nabla^2 f(x) d.
\end{equation}
At each iteration, \arcq minimizes a cubic model~\citep{Gr81}
\begin{equation}
  \label{eq:ARC}
  c_x^{\alpha}(d) :=  q_x(d) + \tfrac{1}{3} \alpha^{-1} \|d\|^3,
\end{equation}
where \(\alpha > 0\) plays a role similar to the trust-region radius in TR
methods.
%
%
As in TR methods, minimizing~\eqref{eq:ARC} involves solving the shifted linear system
\begin{equation}
  \label{eq:LevMar}
  (\nabla^2 f(x) + \lambda I) d = -\nabla f(x),
\end{equation}
while searching an appropriate value of the shift \(\lambda > 0\).
\citet{cartis-gould-toint-2011a} use GLTR~\citep{gltr}, which builds upon the truncated conjugate gradient approach of~\citet{steihaug-1983} and~\citet{toint-1981} and further explores the boundary of the trust region if the latter is determined to be active.
One disadvantage of GLTR resides in its storage and computational requirements as Lanczos vectors must be either stored or re-generated to explore the trust-region boundary.
This additional cost, which is incurred from the very first iteration, makes the approach unsuitable for large-scale problems.
In addition, numerical experiments reveal that the method of~\citet{steihaug-1983} and~\citet{toint-1981} yields an adequate solution most of the time, so that further explorations of the boundary of the trust region may not pay off, as already remarked by the authors of GLTR\@.


In this paper, we propose \arcqk, an implementation of \arcq that obtains
an approximate minimizer of~\eqref{eq:ARC} using CG-Lanczos with shifts, a Lanczos implementation of the conjugate gradient algorithm that solves several shifted systems simultaneously proposed by~\citet{frommer-maass-1999}.
At each iteration of \arcqk, a range of \(m \ll n\) shift parameters is selected.
The CG variant solves the shifted systems concurrently and interrupts those corresponding to shifts that are too small as soon as negative curvature is detected.
The computational overhead beyond the Lanczos process is thirteen scalar operations to update five vectors of length \(m\) and two \(n\)-vector updates for each value of the shift.

Of course, the overhead of using some 31 shift will not be easily overcomed for small scaled instances. Our results, summarized in performance profile graphs, confirms the efficiency of our approach for problems of dimensions above 100 in the Cutest collection when compared to a Steihaug-Toint approach.

The rest of this paper is organized as follows. We focus the presentation on the worst case function evaluation complexity to achieve first order optimality conditions. We recall the \arcq algorithm and its worst case evaluation complexity analysis in \S\ref{sec:AA}. We introduce \arcqk and analyze its worst case complexity. We then introduce the CG-Lanczos method to solve the shifted systems and analyze its computational complexity. We report on numerical experience on problems from the \textsf{CUTEst} collection~\citep{gould-orban-toint-2015}.

\subsection*{Notation}

When the context is clear, we simply write \(q(d)\) and \(c^{\alpha}(d)\) instead of \(q_x(d)\) and \(c_x^{\alpha}(d)\).
In an iterative procedure, we write \(q_k(d)\) and \(c_k(d)\) instead of \(q_{x_k}(d)\) and \(c_{x_k}^{\alpha_k}(d)\).
Throughout the text, we use \(\|\cdot\|\) to denote the Euclidean norm.
For future reference, the gradient of~\eqref{eq:ARC} is
\begin{equation}
 \label{eq:ARC-grad}
  \nabla c^{\alpha}_x(d) = \nabla f(x) + \nabla^2 f(x) d + \alpha^{-1} \|d\| d.
\end{equation}
If \(A\) is a square symmetric matrix, we write \(A \succ 0\) to mean that \(A\) is positive definite.
If \(\{\alpha_k\}\) and \(\{\beta_k\}\) are real positive sequences converging to zero, we use the Landau notation \(\alpha_k = O(\beta_k)\), \(\alpha_k = \Omega(\beta_k)\) and \(\alpha_k = \Theta(\beta_k)\) to mean that there exist a constant \(C > 0\) and an iteration \(k_0 \in \N\) such that \(\alpha_k \leq C \beta_k\), \(\alpha_k\geq C^{-1} \beta_k\), and \(C^{-1} \beta_k \leq \alpha_k \leq C \beta_k\) for all \(k \geq k_0\), respectively.

\section{The \texorpdfstring{\arcq}{ARCq} algorithm}%
\label{sec:AA}

The basic algorithm, described as \Cref{alg:ARCq}, is similar to the basic TR method; \arcq uses the cubic regularized model to compute the direction \(d\) but the quadratic model in the algorithm flow.

The following result characterizes global minimizers of \(c_k(d)\).

\begin{theorem}[\protect{\citealp[Theorem 3.1]{cartis-gould-toint-2011a}}]%
\label{M-R-comp}
Any \(d_k\) is a global minimizer of \(c_k(d)\) if and only if
\begin{subequations}
  \begin{align}
    \nabla f(x_k) + (\nabla^2 f(x_k) + \lambda_k I) d_k & = 0,
    \label{eq:Gmind}
    \\
    \nabla^2 f(x_k) + \lambda_k I & \succeq 0,
    \label{eq:SDP}
    \\
    \lambda_k & = \|d_k\| / \alpha_k.
    \label{eq:Gmind2}
  \end{align}
\end{subequations}
If \(\nabla^2 f(x_k) + \lambda_k I  \succ 0\), then \(d_k\) is unique.
\end{theorem}

\smallskip

\Cref{M-R-comp} suggests a computational procedure similar to the~\citet{more-sorensen-1983} approach in trust-region methods.
\citet{dussault-2015} improves the procedure by using changes of variables to reduce the number of factorizations to one per successful iteration.
\Cref{alg:ARCq} summarizes the \arcq framework.

\begin{algorithm}[htbp]
  \caption{%
    \label{alg:ARCq}
    \arcq Framework.
 }
 \begin{algorithmic}[1]
   \State Initialize \(x_0 \in \R^n\), \(\alpha_0 > 0\), \(0<\eta_1<\eta_2<1\), and \(0<\gamma_1<1<\gamma_2\), \(k = 0\)
   \Repeat
     \State\label{line:globalmin} compute \(d_k\) as a global minimizer of \(c_k\)
     \State compute \(\rho_k = \dfrac{f(x_k) - f(x_k + d_k)}{q_k(0) - q_k(d_k)}\)
     \If {\(\rho_k < \eta_1\)}
       \State \(\alpha_{k+1} = \gamma_1 \alpha_k\) \Comment{Unsuccessful iteration}
     \Else
       \State \(x_{k+1} = x_k + d_k\) \Comment{Successful iteration}
       \If {\(\rho_k > \eta_2\)}
         \State \(\alpha_{k+1} = \gamma_2 \alpha_k\)  \Comment{Very successful iteration}
       \Else
         \State \(\alpha_{k+1} = \alpha_k\)
       \EndIf
     \EndIf
     \State \(k \gets k + 1\)
  \Until  termination\_criterion
  \end{algorithmic}
 \end{algorithm}

\Cref{alg:ARCq} differs from the original ARC of~\citet{cartis-gould-toint-2011a} only in the use of the quadratic model \(q_k\) to compute the ratio \(\rho\) instead of the cubic model \(c_k\). The convergence and complexity results are identical and rely on the following assumptions.


\begin{assumption}%
  \label{asm:C2}
  \(f\) is twice continuously differentiable.
\end{assumption}

\begin{assumption}%
  \label{asm:d2f-lipschitz}
  There exists \(L > 0\) such that \(\|\nabla^2 f(x)-\nabla^2f(y)\|\le L\|x-y\|\) for all \(x\), \(y \in \R^n\).
\end{assumption}

\begin{assumption}%
  \label{asm:lbnd}
  There exists a value \(f_{\text{low}}\) such that \(f(x_k) \ge f_{\text{low}}\) for all \(k\).
\end{assumption}

\begin{assumption}%
  \label{asm:globalmin}
  \(d\) is computed as a global minimizer of \(c_x^{\alpha}\).
\end{assumption}

\begin{assumption}%
  \label{asm:alpha}
  There exists \(\alpha_{\max}>0\) such that \(\alpha_k\le\alpha_{\max}\) for all \(k\).
\end{assumption}

Under the above assumptions,~\citet{dussault-2015} establishes the following global convergence result.

\begin{theorem}[\protect{\citealp[Corollary~\(3.3\)]{dussault-2015}}]%
\label{th:ConvCNO}
Let \(\{x_k\}\) be generated by \Cref{alg:ARCq} and let \Crefrange{asm:C2}{asm:alpha} be satisfied.
Any cluster point of \(\{x_k\}\) satisfies the second-order necessary optimality conditions.
\end{theorem}

A further property related to Newton's method applied to~\eqref{eq:Gmind}--\eqref{eq:Gmind2} yields the following complexity result.

\begin{theorem}[\protect{\citealp[Theorem~\(4.2\)]{dussault-2015}}]%
  \label{th:CplxARCq}
  Under \Crefrange{asm:C2}{asm:alpha}, the number of iterations required by \Cref{alg:ARCq} to generate \(\bar x\) such that \(\|\nabla f(\bar x)\|<\epsilon\) is at most \(O(\epsilon^{-3/2})\).
\end{theorem}

\Cref{th:ConvCNO} and \Cref{th:CplxARCq} rely on exact minimization of the cubic model~\eqref{eq:ARC}.
The analysis involving global minimizers of~\eqref{eq:ARC} ensures that \(\Delta q_k(d_k):= q_k(d_k)- q_k(0) \ge \tfrac{1}{2} \lambda_k \|d_k\|^2\) and that \(\|\nabla f(x_{k+1})\| \le \tfrac{1}{2} L \|d_k\|^2 + \lambda_k \|d_k\|\).
The proof of \Cref{th:CplxARCq} relies on the fact that both \(\lambda_k = \Omega(\|d_k\|)\)
 and  \(\lambda_k = O(\|d_k\|)\) hold if \(\alpha_k\) is bounded away from zero and \(d_k\) is computed as a global minimizer of~\eqref{eq:ARC}, for in that case \(\lambda_k = \|d_k\| / \alpha_k\).

In large-scale applications, we must instead consider approximate solutions to~\eqref{eq:ARC}, the analysis of which is the subject of the next section.

\subsection{Approximate model solution}%
\label{sec:ApproxCubic}

In this section, we show that the complexity bound of \Cref{th:CplxARCq} continues to hold for approximate minimizers if we can guarantee that \(\alpha_k\) is bounded away from zero and \(\lambda_k = \Theta(\|d_k\|)\).
\Cref{le:alphaMinK} establishes the first property.
In \Cref{sec:Lanczos}, we propose a way to compute \(\lambda_k\) and \(d_k\) so that \(\lambda_k = \Theta(\|d_k\|)\) as well as \(d_k^T (\nabla^2 f(x) + \lambda_k I) d_k \geq 0\).


For the time being, we require \(\lambda_k\) and \(d_k\) to satisfy the following approximation to~\eqref{eq:Gmind}--\eqref{eq:Gmind2}:
\begin{subequations}
  \begin{align}
    \nabla f(x_k) + (\nabla^2 f(x_k) + \lambda_k I) d_k & = r_k,
    \label{eq:Kr1} \\
    d_k^T (\nabla^2 f(x_k) + \lambda_k I) d_k & \ge 0
    \label{eq:Kr2} \\
    \frac1{\beta} \frac{\|d_k\|}{\alpha_k} \le
    \lambda_k & \le \beta \frac{\|d_k\|}{\alpha_k},
    \label{eq:Kr3}
  \end{align}
\end{subequations}
where \(r_k\) is a residual and \(\beta \ge 1\) is a sampling parameter.
%



We modify \Cref{asm:globalmin} as follows.

\begin{assumptionb}{asm:globalmin}%
  \label{asm:approxmin}
  At line~\ref{line:globalmin} of \Cref{alg:ARCq}, we compute a direction \(d\) satisfying~\eqref{eq:Kr1}--\eqref{eq:Kr3}.  
\end{assumptionb}

In \Cref{sec:convergence}, we state appropriate conditions on \(\|r\|\) that ensure convergence and maintain the complexity bound of \Cref{th:CplxARCq}.

\subsection{Convergence analysis and complexity}%
\label{sec:convergence}

We are now able to generalize the basic lemmas from~\citet{dussault-2015} to approximate minimizers of~\eqref{eq:ARC}.

\begin{lemma}\label{le:deltaqmK}
  Let \Cref{asm:C2,asm:approxmin} be satisfied and assume that \(r_k^T d_k = 0\). Then,
  \[
    \Delta q_k(d_k) =
    q_k(0) - q_k(d_k) \geq
    \tfrac{1}{2} \, \|d_k\|^{2} \lambda_k \ge
    \frac{\|d_k\|^{3}}{2 \beta \alpha_k}.
  \]
\end{lemma}

\begin{proof}
We multiply~\eqref{eq:Kr1} by \(d_k\) and use the assumption that \(r_k^T d_k = 0\) to obtain 
\[
  - \nabla f(x_k)^T d_k = d_k^T (\nabla^2 f(x_k)+\lambda_k I) d_k,
\]
so that
\begin{align}
  q_k(0) - q_k(d_k) &= -\nabla f(x_k)^T d_k - \tfrac{1}{2} d_k^T \nabla^2 f(x_k) d_k
  \\ & = \tfrac{1}{2} d_k^T \nabla^2 f(x_k) d_k + \lambda_k \|d_k\|^{2},\label{eq:qcbound}
\end{align}
and using~\eqref{eq:Kr2} and~\eqref{eq:Kr3} yields the result.
\end{proof}

The assumption that \(r_k^T d_k = 0\) will naturally be satisfied by the implementation that we propose in \Cref{sec:Lanczos} and corresponds to the requirement~\cite[(3.11)]{cartis-gould-toint-2011a}.

We next observe that \(\alpha_k\) is bounded away from zero if \(\nabla^2 f\) is Lipschitz continuous.

\begin{lemma}\label{le:alphaMinK}
If \Cref{asm:C2,asm:d2f-lipschitz,asm:approxmin} are satisfied and \(r_k^T d_k = 0\), then \(\alpha_{k+1} \ge \alpha_k\) whenever \(\alpha_k < (1-\eta_2) / (2\beta L)\). Thus, \(\alpha_k \ge \alpha_{\min} > 0\)  for all \(k \ge 0\) where \(\alpha_{\min} := \min(\alpha_0, \, \gamma_1 (1-\eta_2) / (2\beta L))\).
\end{lemma}

\begin{proof}
\Cref{asm:d2f-lipschitz} ensures that \(f(x_k+d_k)-q_k(d_k)\le L\|d_k\|^3\).
\Cref{le:deltaqmK} implies
\[
  \rho_k =
    \frac{f(x_k)-f(x_k+d_k)}{q_k(0)-q_k(d_k)} =
    1 + \frac{q_k(d_k)-f(x_k+d_k)}{\Delta q_k(d_k)} \geq
    1 - 2 L \beta \alpha_k.
\]
Therefore, \(\rho_k > \eta_2\) whenever \(\alpha_k < (1-\eta_2) / (2\beta L)\), and \(\alpha_{k+1} \ge \alpha_k\).
The smallest possible value to which \(\alpha_k\) may decrease is then \(\gamma_1 (1-\eta_2) / (2\beta L)\).
Should the initial value \(\alpha_0\) be smaller than this threshold, the first iterations will increase \(\alpha_k\) so that the overall lower bound is the value \(\alpha_{\min}\) stated.
\end{proof}

The next requirement states the accuracy on the residual of the Newton equations.%


\begin{assumption}%
  \label{asm:resid}
  In~\eqref{eq:Kr1}, \(\|r_k\|\le \xi \, \min(\|\nabla f(x_k)\|,\|d_k\|)^{2}\) for some \(\xi > 0\).
\end{assumption}

\begin{lemma}\label{le:quadK}
Let \Cref{asm:C2,asm:d2f-lipschitz,asm:approxmin,asm:resid} be satisfied and assume that \(r_k^T d_k = 0\).
Then \(\|\nabla f(x_{k+1})\|\le \kappa_g^{-1} \|d_k\|^2\) for each successful iteration \(k\), where
\(\kappa_g := (\frac12 L + \beta / \alpha_{\min} + \xi)^{-\frac12}\).
\end{lemma}

\begin{proof}
  Using a generalization of the fundamental theorem of integral calculus~\cite[\S{8.1.2}]{ortega-1990},~\eqref{eq:Kr1} and \Cref{asm:C2,asm:d2f-lipschitz}, we may write
  \begin{align*}
    \|\nabla f(x_{k+1})\| & =
    \left\| \nabla f(x_k) + \int_0^1 \nabla^2 f(x_k+\tau d_k) d_k \, \mathrm{d}\tau \right\|
    \\ & =
    \left\|\int_0^1 \nabla^2 f(x_k+\tau d_k) d_k \, \mathrm{d}\tau - (\nabla^2 f(x_k)+\lambda_k I) d_k + r_k \right\|
    \\ & =
    \left\|\int_0^1 (\nabla^2 f(x_k+\tau d_k) -\nabla^2 f(x_k)) d_k \, \mathrm{d}\tau -\lambda_k d_k + r_k \right\|
    \\ & \le
    \|d_k\| \, \left\|\int_0^1 L\tau d_k \, \mathrm{d}\tau \right\| + \lambda_k \|d_k\| + \|r_k\|.
  \end{align*}
  Now,~\eqref{eq:Kr3}, \Cref{asm:resid} and \Cref{le:alphaMinK} combine with the above to yield
  \[
    \|\nabla f(x_{k+1})\| \le
    (\tfrac{1}{2} L + \beta / \alpha_k + \xi) \, \|d_k\|^{2} \le
    \kappa_g^{-2} \, \|d_k\|^{2}.
    \qedhere
  \]
\end{proof}

For a given \(\epsilon > 0\), we now wish to bound the total work required to reach a first iteration \(k(\epsilon)\) such that \(\|\nabla f(x_{k(\epsilon)})\| < \epsilon\) and \(\|\nabla f(x_j)\| \geq \epsilon\) for \(j < k(\epsilon)\).

Following~\citet{cartis-gould-toint-2011b}, we define the index sets
\begin{align*}
  {\mathcal S} & := \{ k \ge 0 \mid \text{iteration \(k\) is successful or very successful} \}
    \\
    {\mathcal S}(\epsilon) & := \{ k \in{\mathcal S} \mid \|\nabla f(x_{k+1})\| \geq \epsilon \}
    \\
  {\mathcal U} & := \{ k \ge 0 \mid \text{iteration \(k\) is unsuccessful} \}.
\end{align*}
For sets such as \({\mathcal S}\), we denote their cardinality as \(|{\mathcal S}|\).


\begin{theorem}[Complexity bound of \arcq]\label{th:CplxARCqK}
Let \Cref{asm:C2,asm:d2f-lipschitz,asm:lbnd,asm:approxmin,asm:alpha,asm:resid} be satisfied and assume that \(r_k^T d_k = 0\).
Then, 
\[
  |{\mathcal S}(\epsilon)| \le
  \frac{f(x_0) - f_{\text{low}}}{C} \, \epsilon^{-\frac32} =
  O(\epsilon^{-\frac32}),
  \qquad
  C := \eta_1 \frac{\kappa_g^3}{2 \beta \alpha_{\max}} > 0,
\]
where \(\kappa_g\) is defined in \Cref{le:quadK}.
In addition,
\[ 
  |\mathcal{S}(\epsilon)| + |\mathcal{U}| \le
  \frac{f(x_0) - f_{\text{low}}}{C} \,
  \left(
    1 + \log \left( \frac{\alpha_{\min}}{\alpha_{\max}} \right) / \log \gamma_1
  \right) \, \epsilon^{-\frac32}.
\] 
\end{theorem}

\begin{proof}
Let \(k\in{\mathcal S}(\epsilon)\).
\Cref{le:deltaqmK,le:quadK} combine with \Cref{asm:alpha} to give
\begin{equation}\label{eq:bndwrtepsilon}
  f(x_k) - q_k(d_k) \ge
    \frac{\kappa_g^3}{2\beta\alpha_k} \, \|\nabla f(x_{k+1})\|^{\frac32} \ge
    \frac{\kappa_g^3}{2\beta\alpha_{\max}} \, \epsilon^{\frac32}.
\end{equation}
Because \(k\) is a successful iteration, 
\[
  f(x_k) - f(x_{k+1}) \ge
    \eta_1 (f(x_k) - q_k(d_k)) \ge
    \eta_1 \frac{\kappa_g^3}{2\beta\alpha_{\max}} \, \epsilon^{\frac32} =
    C \, \epsilon^{\frac32}.
\]
We have shown that at every iteration \(k \in {\mathcal S}(\epsilon)\), \(f\) decreases by at least \(C \epsilon^\frac32 > 0\) and \Cref{asm:lbnd} ensures that \(\mathcal{S}(\epsilon)\) must be finite.
Let \(\mathcal{S}(\epsilon) = \{k_1, \dots, k_{\ell(\epsilon)} \}\) with \(k_i < k_{i+1}\) for \(i = 1, \dots, \ell(\epsilon)\).
Because intermediate unsuccessful iterations keep \(f\) constant, the above combines with \Cref{asm:lbnd} to yield
\[
  f(x_0) - f_{\text{low}} \ge
    f(x_0) - f(x_{k_{\ell(\epsilon)}}) \ge
    \ell(\epsilon) \, C \, \epsilon^{\frac32}.
\]
Therefore,
\[
  |\mathcal{S}(\epsilon)| = \ell(\epsilon) \le \frac{f(x_0) - f_{\text{low}}}{C} \, \epsilon^{-\frac32}.
\]
In the most pessimistic scenario, a sequence of unsuccessful iterations reduces \(\alpha_k\) from \(\alpha_{\max}\) down to \(\alpha_{\min}\) after each successful iteration.
Each such sequence has at most \(\log(\alpha_{\min} / \alpha_{\max}) / \log\gamma_1\) iterations since each unsuccessful iteration reduces \(\alpha_k\) by a factor \(\gamma_1\).
The total number of unsuccessful iterations across all such sequences is thus bounded above by
\[
  |\mathcal{U}| \le
    \ell(\epsilon) \,
    \log \left( \frac{\alpha_{\min}}{\alpha_{\max}} \right) / \log\gamma_1.
\]
Finally, the total number of iterations is bounded above by
\[
  |\mathcal{S}(\epsilon)| + |\mathcal{U}| \le
  \frac{f(x_0) - f_{\text{low}}}{C} \,
  \left(
    1 + \log \left( \frac{\alpha_{\min}}{\alpha_{\max}} \right) / \log\gamma_1
  \right) \,
  \epsilon^{-\frac32}.
  \qedhere
\]
\end{proof}

\Cref{th:CplxARCqK} also establishes global convergence of \arcq to first order stationary points. Indeed, for any \(\epsilon>0\), the number of iterations with \(\|\nabla f(x_k)\|>\epsilon\) is finite, so that \(\liminf \|\nabla f(x_k)\| = 0\).
It relies on the strong assumption that we identify approximate second-order solutions of the regularized subproblems as specified by~\eqref{eq:Kr1}--\eqref{eq:Kr3}.
Because our implementation, detailed in \Cref{sec:Lanczos} satisfies that assumption, \Cref{th:CplxARCqK} is sufficient for our purposes.
However, our implementation is intended for large-scale problems and \Cref{asm:resid} may be too stringent.
Fortunately, it can be relaxed as follows.

\begin{assumptionb}{asm:resid}%
  \label{asm:relaxedresid}
  In~\eqref{eq:Kr1}, \(\|r_k\|\le \xi \, \min(\|\nabla f(x_k)\|,\|d_k\|)^{1 + \zeta}\) for some \(\xi > 0\) and \(0 < \zeta \leq 1\).
\end{assumptionb}

\Cref{asm:relaxedresid} changes \Cref{le:quadK} as follows.

\begin{lemma}%
  \label{le:quadK-relaxedresid}
  Let \Cref{asm:C2,asm:d2f-lipschitz,asm:approxmin,asm:relaxedresid} be satisfied and assume that \(r_k^T d_k = 0\).
  Then \(\|\nabla f(x_{k+1})\| \le \kappa_g^{-1} (\|d_k\|^2 + \|d_k\|^{1 + \zeta})\) for each successful iteration \(k\), where
  \(\kappa_g :=  2 \, \max(\frac12 L + \beta / \alpha_{\min}, \, \xi)\).
\end{lemma}

\begin{proof}
  Using the same arguments as in the proof of \Cref{le:quadK},
  \begin{align*}
    \|\nabla f(x_{k+1})\| & \leq
    \|d_k\| \, \left\|\int_0^1 L\tau d_k \, \mathrm{d}\tau \right\| + \lambda_k \|d_k\| + \|r_k\|
    \\ & \leq
    (\tfrac{1}{2} L + \beta / \alpha_{\min}) \, \|d_k\|^{2} + \xi \, \|d_k\|^{1 + \zeta}.
    \qedhere
  \end{align*}
\end{proof}

Accordingly, \Cref{th:CplxARCqK} is modified as follows.

\begin{corollary}[Complexity bound of approximate \arcq]\label{cor:CplxARCqK}
Under the same assumptions as \Cref{th:CplxARCqK} with \Cref{asm:resid} replaced with \Cref{asm:relaxedresid}, we have
\[
  |{\mathcal S}(\epsilon)| \le
  \frac{f(x_0) - f_{\text{low}}}{C} \, \max(\epsilon^{-3/(1 + \zeta)}, \, \epsilon^{-3/2}),
  \qquad
  C := \eta_1 \frac{\kappa_g^3}{16 \beta \alpha_{\max}} > 0,
\]
where \(\kappa_g\) is defined in \Cref{le:quadK-relaxedresid}.
In addition,
\[
  |\mathcal{S}(\epsilon)| + |\mathcal{U}| \le
  \frac{f(x_0) - f_{\text{low}}}{C} \,
  \left(
    1 + \log \left( \frac{\alpha_{\min}}{\alpha_{\max}} \right) / \log \gamma_1
  \right) \,
  \max(\epsilon^{-3/(1 + \zeta)}, \, \epsilon^{-3/2}).
\]
\end{corollary}

\begin{proof}
  Proceeding as in \Cref{th:CplxARCqK}, for any \(k \in \mathcal{S}(\epsilon)\), \Cref{le:deltaqmK} and \Cref{asm:alpha} imply \(f(x_k) - q_k(d_k) \geq (2 \beta \alpha_k)^{-1} \|d_k\|^3 \geq (2 \beta \alpha_{\max})^{-1} \|d_k\|^3\).
  There are now two cases to consider.
  Consider first the possibility that at this iteration \(k\), \(\|d_k\| \leq 1\) so that \(\|d_k\|^2 \leq \|d_k\|^{1 + \zeta}\).
  We obtain from \Cref{le:quadK-relaxedresid} that \(\|\nabla f(x_{k+1})\| \leq 2 \kappa_g^{-1} \|d_k\|^{1 + \zeta}\).
  We now simply follow the logic of the proof of \Cref{th:CplxARCqK} and conclude that for this iteration,
  \[
    f(x_k) - f(x_{k+1}) \geq \frac{\eta_1 \kappa_g^3}{16 \beta \alpha_{\max}} \, \epsilon^{3/(1 + \zeta)}.
  \]
  Consider next the case \(\|d_k\| > 1\).
  The same reasoning leads us to conclude that for this iteration,
  \[
    f(x_k) - f(x_{k+1}) \geq \frac{\eta_1 \kappa_g^3}{16 \beta \alpha_{\max}} \, \epsilon^{3/2}.
  \]
  Thus, for any \(k \in \mathcal{S}(\epsilon)\),
  \[
    f(x_k) - f(x_{k+1}) \geq C \, \min(\epsilon^{3/(1 + \zeta)}, \, \epsilon^{3/2}).
  \]
  Continuing with the logic of the proof of \Cref{th:CplxARCqK}, we obtain
  \[
    |\mathcal{S}(\epsilon)| \leq \frac{f(x_0) - f_{\text{low}}}{C} \, \max(\epsilon^{-3/(1 + \zeta)}, \, \epsilon^{-3/2}).
  \]
  The remainder of the proof is unchanged.
\end{proof}

In a typical scenario, the tolerance \(\epsilon < 1\) and \Cref{cor:CplxARCqK} yields a complexity of \(O(\epsilon^{-3/(1 + \zeta)})\) iterations.
However, if a loose tolerance \(\epsilon \geq 1\) is of interest, the complexity is instead of \(O(\epsilon^{-3/2})\) iterations.
In the next section, we confirm that in the local regime, superlinear convergence occurs at a rate \(1 + \zeta\).

\subsection{Asymptotic analysis}

We now address the behavior of the \Cref{alg:ARCq} in the vicinity of an isolated local minimizer, i.e., we make the following assumption.

\begin{assumption}%
  \label{asm:2nd-order}
  Assume that \(\{x_k\} \to x^*\) with \(\nabla f(x^*) = 0\) and \(\nabla^2f(x^*) \succ 0\).
\end{assumption}

Our next result states that \(\{\nabla f(x_k)\}\) converges to zero superlinearly with rate \(1 + \zeta\), where \(\zeta\) is stated in \Cref{asm:relaxedresid}.

\begin{theorem}[Superlinear convergence of \arcq]\label{th:SupLinARCqK}
  Let \Cref{asm:C2,asm:d2f-lipschitz,asm:lbnd,asm:approxmin,asm:alpha,asm:relaxedresid,asm:2nd-order} be satisfied and assume that \(r_k^T d_k = 0\) for all sufficiently large \(k\).
  Then there exists an iteration \(K\) such that for all \(k \ge K\),
  \begin{equation}
    \label{eq:arcq-super}
    \|\nabla f(x_{k+1})\| = O(\|\nabla f(x_k)\|^{1+\zeta}).
  \end{equation}
\end{theorem}

\begin{proof}
  Let \(\lambda_{\min}^*>0\) denote the smallest eigenvalue of \(\nabla^2 f(x^*)\).
  \Cref{asm:d2f-lipschitz} ensures that there is \(K \in \N\) such that the smallest eigenvalue of \(\nabla^2 f(x_k)\) is \(\lambda_{\min}^k \geq \tfrac{1}{2} \lambda_{\min}^*\) for all \(k \geq K\).

  Proceeding as in the proof of \Cref{le:deltaqmK},~\eqref{eq:qcbound} yields \(\Delta q_k(d_k) \ge \frac{1}{2}\lambda_{\min}^k\|d_k\|^2\).
  A similar analysis to that of the proof of \Cref{le:alphaMinK} then implies
  \begin{equation}
    \label{eq:rho-asympt}
    \rho_k \ge
    1 - 2L \frac{\|d_k\|}{\lambda_{\min}^k} \ge
    1 - 4L \frac{\|d_k\|}{\lambda_{\min}^*}.
  \end{equation}
  Because \(\{x_k\} \to x^*\), we can assume that \(K\) is large enough that for all \(k \geq K\),
  \[
    \|d_k\| \leq \frac{1 - \eta_1}{4L} \lambda_{\min}^*,
  \]
  which combines with~\eqref{eq:rho-asympt} to show that \(\rho_k \ge \eta_1\) for all \(k \geq K\), and thus, only successful iterations are performed.

  Now, using~\eqref{eq:Kr1} and \Cref{asm:relaxedresid},
  \begin{align*}
    \|d_k\| & \leq \|(\nabla^2 f(x_k) + \lambda_k I)^{-1}\| \, \|r_k - \nabla f(x_k)\|
    \\      & \leq \frac{\|r_k\| + \|\nabla f(x_k)\|}{\lambda_{\min}^k}
    \leq 2 \|\nabla f(x_k)\| \, \frac{1 + \|\nabla f(x_k)\|^\zeta}{\lambda_{\min}^*},
  \end{align*}
  for all \(k \geq K\).
  Finally, we can also assume that \(\|d_k\| \leq 1\) so that \Cref{le:quadK-relaxedresid} yields~\eqref{eq:arcq-super}.
\end{proof}

\section{CG-Lanczos implementation}%
\label{sec:Lanczos}

We now present an implementation that takes advantage of the approximation criteria developed above.
Our implementation is based on the conjugate-gradient method (CG) of \citet{hestenes-stiefel-1952}.
We first explain how we use shifted systems to implement \arcqk and then describe the details of our implementation of CG based on the~\citet{lanczos-1950} process to solve all the shifted systems simultaneously.
Although the Lanczos form of CG differs from that initially stated by \citet{hestenes-stiefel-1952}, it is equivalent to it in exact arithmetic, can be derived from the Lanczos tridiagonalization process, and is amenable to the simultaneous solution of shifted symmetric systems.

The key idea is to discretize the half line \(0<\lambda<\infty\) into values \(0< \lambda_0 < \cdots < \lambda_m < \infty\).
For reasons motivated by double precision arithmetic, we impose \(10^{-15} \le \lambda_i \le 10^{15}\) for \(i = 0, \dots, m\).
From the formulation~\eqref{eq:Kr3}, a simple choice consists in setting \(\beta\lambda_{i} = \lambda_{i+1} / \beta\), for some \(\beta > 1\).
For instance, \(\beta = \sqrt{10}\) yields the \(31\) values \(\lambda_i = 10^i\) for \(i = -15, \dots, 15\). The computational feasibility of such a procedure is detailed below and comes from the fact that an appropriate shifted CG-Lanczos implementation obtains all \(m+1\) (approximate) solutions of
\(
  (\nabla^2 f(x) + \lambda_i I) d(\lambda_i)
  \approx
  -\nabla f(x),
\)
or establishes that \(\nabla^2 f(x)+\lambda_i I \not \succeq 0\), at modest expense beyond that of solving a single linear system.

We now describe how we solve a sequence of shifted linear systems simultaneously in a way that is consistent with the minimization of~\eqref{eq:ARC}. Our implementation is an adaptation of \citep[Algorithm 6]{frommer-maass-1999}.

\Cref{alg:slcg} describes the CG-Lanczos with shifts implementation for a generic symmetric system \(Mx = b\) with shifts \(\lambda_i\), i.e.,
\begin{equation}
  \label{eq:shifted-sys}
  (M + \lambda_i I) x = b, \quad
  i = 0, \dots, m.
\end{equation}
In \Cref{alg:slcg}, boldface quantities are block quantities with one component per shift parameter.
Initialization statements initialize all \(m+1\) values of a given block variable to identical copies of the right hand side value.
For instance, the statement \(\bm{p} = b\) means that the \(n \times (m+1)\) array \(\bm p\) is initialized to \(m+1\) copies of \(b\).
The statement \(\bm{\sigma} = \beta\) means that all \(m+1\) elements of the array \(\bm{\sigma}\) are initialized to \(\beta\).
For conciseness, the shifts are gathered in the array \(\bm{\lambda}\).

\begin{algorithm}[htbp]
  \caption{Lanczos-CG with shifts for~\eqref{eq:shifted-sys}}%
  \label{alg:slcg}
  \begin{algorithmic}[1]
    \State Set \(\bm{x}_0 = 0\), \(\beta_0 v_0 = b\), \(\bm{p}_0 = b\) \Comment{\(\beta_0\) such that \(\|v_0\| = 1\)}
    \State set \(v_{-1} = 0\), \(\bm{\sigma}_0 = \beta_0\), \(\bm{\omega}_{-1} = 0\),
      \(\bm{\gamma}_{-1} = 1\), \(\bm{\pi}_0 = \beta_0^2\) \Comment{\(\bm{\pi}_0 = \|\bm{p}_0\|^2\)}
    \For{\(j \gets 0, 1, 2, \dots\)}

      \State \(\delta_j = v_j^T M v_j\)
        \Comment{Lanczos part of the iteration}
        \State \(\beta_{j+1} v_{j+1} = M v_j - \delta_j v_j - \beta_j v_{j-1}\) \Comment{\(\beta_{j+1}\) such that \(\|v_{j+1}\| = 1\)}
      \State \(\bm{\delta}_j = \delta_j + \bm{\lambda}\)
        \Comment{CG part of the iteration in block form}
      \State \(\bm{\gamma}_j = 1 / (\bm{\delta}_j - \bm{\omega}_{j-1} / \bm{\gamma}_{j-1})\)
      \State \(\bm{\omega}_j = (\beta_{j+1} \bm{\gamma}_j)^2\)
      \State \(\bm{\sigma}_{j+1} = -\beta_{j+1} \bm{\gamma}_j \bm{\sigma}_j\) \Comment{\(\bm{\sigma}_{j+1} = \|\bm{r}_{j+1}\|\)}
      \State \(\bm{x}_{j+1} = \bm{x}_j + \bm{\gamma}_j \bm{p}_j\)
      \State \(\bm{p}_{j+1} = \bm{\sigma}_{j+1} v_{j+1} + \bm{\omega}_j \bm{p}_j\)
      \State\label{alg:slcg:update-pi}%
        \(\bm{\pi}_{j+1} = \bm{\sigma}_{j+1}^2 + \bm{\omega}_j^2  \bm{\pi}_j\) \Comment{\(\bm{\pi}_{j+1} = \|\bm{p}_{j+1}\|^2\)}
    \EndFor
  \end{algorithmic}
\end{algorithm}

A few observations about \Cref{alg:slcg} are in order. Firstly, note
that a single operator-vector product is required per iteration, and takes
place in the Lanczos part of the iteration, which is independent of the shifts.
The extra cost incurred by requesting the solution of multiple shifted systems
is confined to the CG part of the iteration, which only performs scalar and
vector operations.

Secondly, recall that the vectors \(v_j\) are orthonormal in exact arithmetic
while the search directions \(\bm{p}_j\) are \((M + \bm{\lambda} I)\)-conjugate so
long as negative curvature is not detected.

Thirdly, \Cref{alg:slcg} neither forms nor recurs the residual \(\bm{r}_j = b - M \bm{x}_j\).
A recursion argument shows that \(\bm{r}_j = \bm{\sigma}_j v_j\), and by orthogonality, \(\|\bm{r}_j\| = \bm{\sigma}_j\) is available at no extra cost and may be used in a stopping criterion.

Finally, \Cref{alg:slcg} uses \(\bm{\pi}_j\) to recur \(\|\bm{p}_j\|^2\), i.e., the squared Euclidean norm of each column of \(\bm{p}_j\).
Initially, each component of \(\bm{\pi}_0\) is set to \(\beta_0^2 = \|b\|^2\).
Because the conjugate-gradient method minimizes a quadratic whose gradient at \(\bm{x}_j\) is \(-\bm{r}_j\) along the direction \(\bm{p}_j\) using an exact linesearch, we must have \(\bm{r}_{j+1}^T \bm{p}_j = \bm{\sigma}_{j+1} v_{j+1}^T \bm{p}_j = 0\).
Thus,
\[
  \|\bm{p}\|_{j+1}^2 =
  \bm{\sigma}_{j+1}^2 \|v_{j+1}\|^2 + 2 \bm{\sigma}_{j+1} \bm{\omega}_j v_{j+1}^T \bm{p}_j + \bm{\omega}^2 \|\bm{p}_j\|^2,
\]
which yields the update on line~\ref{alg:slcg:update-pi}.

Not all shifted systems will require the
same number of iterations to converge and we interrupt iterations corresponding to values
of the shift for which either the required tolerance is reached, or negative
curvature is detected.


We now describe how negative curvature may be detected during the iterations of
\Cref{alg:slcg}. Because the argument is independent of the shift, we
assume that \(m = 1\) and \(\lambda_1 = 0\), i.e., we solve the system \(Mx = b\)
with the Lanczos variant of CG\@. At iteration \(j\),
\[
  \delta_j = v_j^T M v_j,
\]
where the vectors \(\{v_j\}\) are orthonormal. If negative curvature is present, \(\delta_j\) may never reveal so, but \(p_j^T M p_j\) will.
We seek a cheap expression to check the sign of \(p_j^T M p_j\) where the vectors \(\{p_j\}\) are \(M\)-conjugate.
At iteration \(j\), \(p_j = r_j + \omega_{j-1} p_{j-1}\), where \(r_j = b - M x_j\) is updated cheaply via \(r_j = \sigma_j v_j\), and \(\omega_{j-1}\) and \(\sigma_j\) are scalars.
Thus
\begin{align}
  p_j^T M p_j
  & = p_j^T M r_j + \omega_{j-1} p_j^T M p_{j-1} \nonumber
  \\
  & = p_j^T M r_j \nonumber
  \\
  & = r_j^T M r_j + \omega_{j-1} p_{j-1}^T M r_j \nonumber
  \\
  & = \sigma_j^2 \delta_j + \omega_{j-1} p_{j-1}^T M r_j. \label{eq:pMp}
\end{align}
The iterates are updated according to \(x_j = x_{j-1} + \gamma_{j-1} p_{j-1}\),
so that
\[
  r_j =
  b - M x_j =
  b - M x_{j-1} - \gamma_{j-1} M p_{j-1} =
  r_{j-1} - \gamma_{j-1} M p_{j-1}.
\]
By orthogonality,
\[
  \sigma_j^2 =
  r_j^T r_j =
  r_j^T r_{j-1} - \gamma_{j-1} r_j^T M p_{j-1} =
  -\gamma_{j-1} r_j^T M p_{j-1},
\]
and therefore,
\begin{equation}
  \label{eq:pMr}
  p_{j-1}^T M r_j =
  -\frac{1}{\gamma_{j-1}} r_j^T r_j =
  -\frac{1}{\gamma_{j-1}} \sigma_j^2.
\end{equation}
Finally, using the update formula for \(\gamma_j\),~\eqref{eq:pMp} and~\eqref{eq:pMr} combine to yield
\[
  p_j^T M p_j =
  \sigma_j^2 (\delta_j - \omega_{j-1} / \gamma_{j-1}) =
  \sigma_j^2 / \gamma_j.
\]
Therefore the sign of \(\gamma_j\) is the same as that of \(p_j^T M p_j\).


\section{Implementation and numerical experiments}%
\label{sec:NumExp}

We implemented \Cref{alg:ARCqK} in the Julia language version~\(1.2\) and used the following packages from the JuliaSmoothOptimizers organization~\citep{jso-2019}:
\begin{itemize}
  \item NLPModels.jl provides a consistent optimization problem interface on which solvers can rely regardless of where problems come from;
  \item CUTEst.jl gives solvers access to the \textsf{CUTEst} collection \citep{gould-orban-toint-2015} by bridging it with NLPModels.jl;
  \item Krylov.jl provides an implementation of \Cref{alg:slcg} alongside other Krylov methods;
  \item SolverBenchmark.jl provides data structures to gather problem statistics, and utilities to generate tables of results and performance profiles;
  \item SolverTools.jl implements building blocks for optimization algorithms, including trust region management and linesearch procedures, and facilities for benchmarking solvers on collections of test problems.
\end{itemize}

During the course of \Cref{alg:slcg}, the solution of the \(i\)th system \((\nabla^2 f(x_k) + \lambda_i I) d = -\nabla f(x_k)\), \(i = 1, \dots, m\), is interrupted as soon as there is an iteration \(j\) for which the \(i\)th component of \(\bm{\gamma}_j\) is negative, which indicates that \(\nabla^2 f(x) + \lambda_i I \not \succeq 0\).

If \(\bm{\gamma}_j \geq 0\) for all \(j\), the solution of the \(i\)th system is terminated as soon as there is an iteration \(j\) for which the residual norm \(\|r_j\| = \|\nabla f(x_k) + (\nabla^2f(x_k) + \lambda_i I) d_j\|\) satisfies \Cref{asm:relaxedresid}.

The complete procedure is summarized as \Cref{alg:ARCqK}.
At line~\ref{def-i+}, we let \(i^+\) be the index of the smallest shift for which no negative curvature was detected.
Thus, no negative curvature was detected for any \(i^+ \leq i \leq m\).
If there exists no such \(i^+\), the algorithm terminates because \(\nabla^2 f(x_k)\) has negative eigenvalues larger than the largest allowed shift in magnitude.
In our implementation, that means negative eigenvalues smaller than \(-10^{15}\).
It is always possible to increase the number and/or the maximum value of the sampled \(\lambda\) if such occurrences happen often.
Such a situation never happened in our numerical experiments.

If \(d_i^* \approx d(\lambda_i)\) denotes the solution thus identified by \Cref{alg:slcg} corresponding to \(\lambda_i\), we retain the one that most closely satisfies \(\alpha_k \lambda_i = \|d_i^*\|\) at line~\ref{def-j}.

When an iteration is unsuccessful, \arcq must compute a minimizer of~\eqref{eq:ARC} with the updated regularization parameter \(\alpha_{k+1} = \gamma_1 \alpha_k\).
By contrast, \arcqk simply turns its attention to the next shift value and the associated search direction, which was already computed.
For a given \(\lambda\), an exact solution \(d(\lambda)\) would satisfy~\eqref{eq:Gmind2}.
If \(\lambda_k\) is the shift used at iteration~\(k\), and the iteration is unsuccessful, we search for the smallest \(\lambda_j > \lambda_k\), i.e., the smallest \(j > k\) such that \(\alpha(\lambda_j) := \|d_j\| / \lambda_j \leq \gamma_1 \alpha_k\).
If this search were to fail, it would indicate that our sample of shifts does not contain sufficiently large values.

\begin{algorithm}[htbp]
  \caption{\arcqk.}%
  \label{alg:ARCqK}
  \begin{algorithmic}[1]
    \State Initialize \(x_0 \in \R^n\), \(\alpha_0 > 0\), \(0<\eta_1<\eta_2<1\), \(0 < \gamma_1 < 1 < \gamma_2\), \(\bm{\lambda}\), \(k = 0\)
    \Repeat
    \State solve \((\nabla^2 f(x_k) + \bm{\lambda} I) d(\bm{\lambda}) = -\nabla f(x_k)\) for \(d(\bm{\lambda})\) \Comment{use \Cref{alg:slcg}}
      \State success \(=\) \CLE{false}
      \State\label{def-i+}%
        \(i^+ = \min \{0\le i\le m \mid \nabla^2 f(x)+\lambda_i I \succ 0\}\)
      \State\label{def-j}%
        \(j = \arg\min \{ |\alpha\lambda_i-\|d(\lambda_i)\|| \mid i^+ \le i \le m\}\)
      \Repeat
      \State  \(d_k = d(\lambda_j)\)
      \State compute \(\rho_k = \dfrac{f(x_k) - f(x_k + d_k)}{q_k(0) - q_k(d_k)}\)
      \If {\(\rho_k < \eta_1\)} \Comment{unsuccessful}
      \State \(\alpha_{k+1} = \alpha_k\)
      \While {\(\alpha_{k+1} > \gamma_1\alpha_k\)}
         \State \(\alpha_{k+1} = \|d_{j+1}\| / \lambda_{j+1}\)  \Comment{consider next shifts}
          \State \(j = j+1\)
      \EndWhile
      \Else \Comment{successful}
        \State success \(=\) \CLE{true}
        \State \(x_{k+1} = x_k + d_k\)
        \If {\(\rho_k > \eta_2\)}  \Comment{very successful}
          \State \(\alpha_{k+1} = \gamma_2 \alpha_k\) \Comment{increase \(\alpha\)}
        \Else
          \State \(\alpha_{k+1} = \alpha_k\)
        \EndIf
      \EndIf
      \Until  success
      \State \(k \gets k + 1\)
    \Until  termination\_criterion
  \end{algorithmic}
\end{algorithm}

In order to assess the scalability of our implementation, we perform numerical experiments using all \(240\) unconstrained problems from the \textsf{CUTEst} collection.
We compare \arcqk with an implementation of the truncated conjugate-gradient trust-region method \citep{steihaug-1983, toint-1981}.
Solving a trust-region subproblem
\[
  \minimize{d} \ q_x(d) \quad \st \ \|d\| \leq \Delta,
\]
where \(\Delta > 0\) is the trust-region radius, involves identifying a shift \(\lambda \geq 0\) satisfying~\eqref{eq:Gmind}--\eqref{eq:SDP} but with~\eqref{eq:Gmind2} replaced with the complementarity condition
\[
  \lambda (\Delta - \|d\|) = 0.
\]
We refer the interested reader to~\citep{conn-gould-toint-2000,steihaug-1983} for more details.
For our comparison, we use a classical Steihaug-Toint approach in which we use a standard non shifted conjugate gradient algorithm which is interrupted either when reaching the boundary of the trust region or when satisfying the required accuracy.

Our experiments ran on a \(2.4\)GHz four-core Intel Xeon E5530 with hyperthreading and \(46\)Gb of memory.
We limit the computing time to \(3,600\) seconds and declare stationarity as soon as \(\|\nabla f(x)\| \leq \epsilon_a + \epsilon_r \|\nabla f(x_0)\|\) with \(\epsilon_a = 10^{-5}\) and \(\epsilon_r = 10^{-6}\).
We do not set limits on the number of function evaluations.

The complete Julia implementation of \Cref{alg:ARCqK} and the Steihaug-Toint variant are described by~\citet{dussault-2019}.
The constants used are \(\zeta=0.5\), \(\gamma_1=0.1\), \(\gamma_2=5.0\), \(\eta_1=0.1\), \(\eta_2=0.75\).
The required accuracy for both the Steihaug-Toint trust region variant and for \arcqk\ is set to \(\|r\|\le \|\nabla f(x)\|^{1+\zeta}\), see \Cref{asm:resid}.

\arcqk solved \(223\) problems successfully, exceeded the time budget on \(14\) problems, declared problem {\tt hielow} unbounded below,\footnote{An iterate was generated where the objective evaluated to \(-\infty\).} and exited for an unknown reason on problem {\tt nelsonls}.
Our Steihaug-Toint implementation solved \(226\) problems, and exceeded the time budget on \(14\) problems.
Both solvers exceeded the time budget on {\tt ba\_l16ls}, {\tt ba\_l21ls}, {\tt ba\_l49ls}, {\tt ba\_l52ls}, {\tt ba\_l73ls}, {\tt eigenbls}, {\tt fletcbv3}, {\tt fletchbv}, {\tt indef}, {\tt jimack}, {\tt nonsmsqrt}, {\tt sbrybnd} and {\tt scosine}.
In addition, \arcqk exceeded the time budget on {\tt eigencls} and {\tt mnists5ls}, and the Steihaug-Toint implementation exceeded the time budget on {\tt yatp2ls}.
Note that most of the above problems are large nonlinear least-squares problems for which Newton's method is probably not appropriate.

Complete numerical results for both algorithms are reported in the electronic supplement accompanying this paper.
\Cref{fig:arc-st} reports aggregated results in the form of time and evaluation \citet{dolan-more-2002} performance profiles.

While our preliminary implementation is slower than the Steihaug-Toint method, and requires more objective value and gradient evaluations, it saves a substantial amount of Hessian-vector products.
Moreover, for a restricted set of the \textsf{CUTEst} collection of all (113) the problems of size at least 100, \arcqk matches or outperforms the Steihaug-Toint method.
This is to be expected, the overhead of solving the shifted systems being relatively high for smaller instances.

\begin{figure}[p]
  \begin{center}
    \includegraphics[height=.185\textheight]{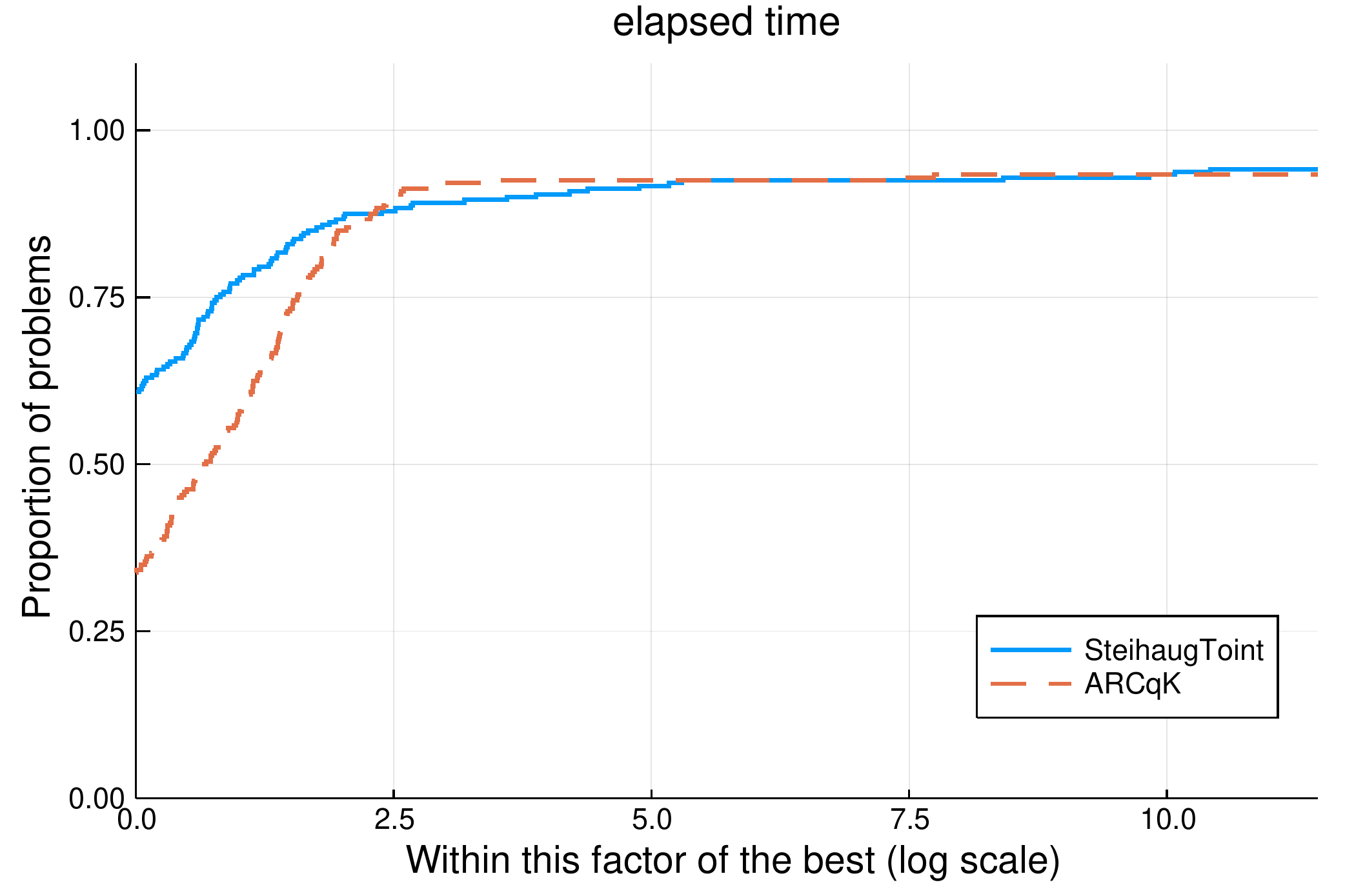}
    \hfill
    \includegraphics[height=.185\textheight]{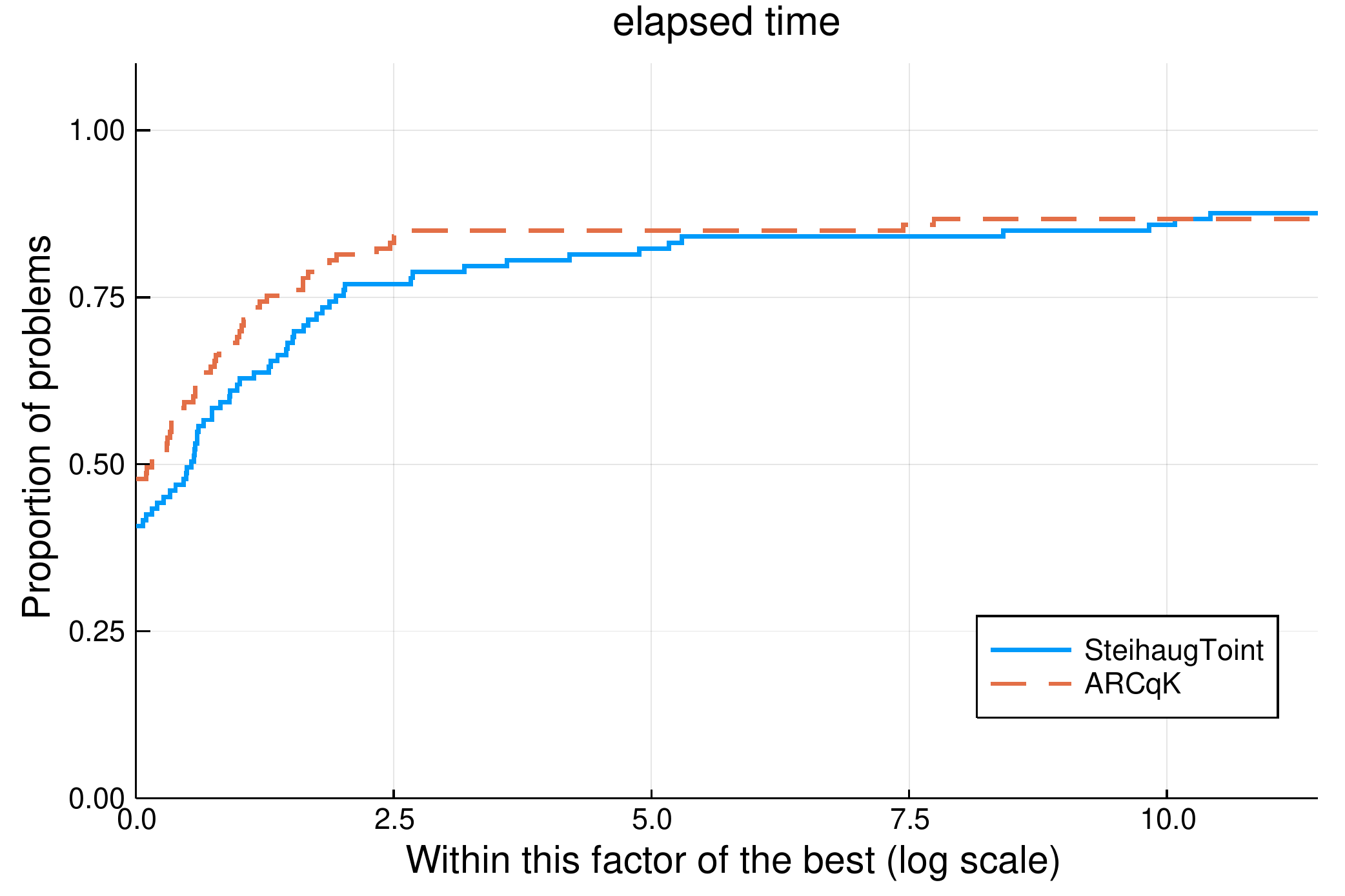}
    \\ 
    \includegraphics[height=.185\textheight]{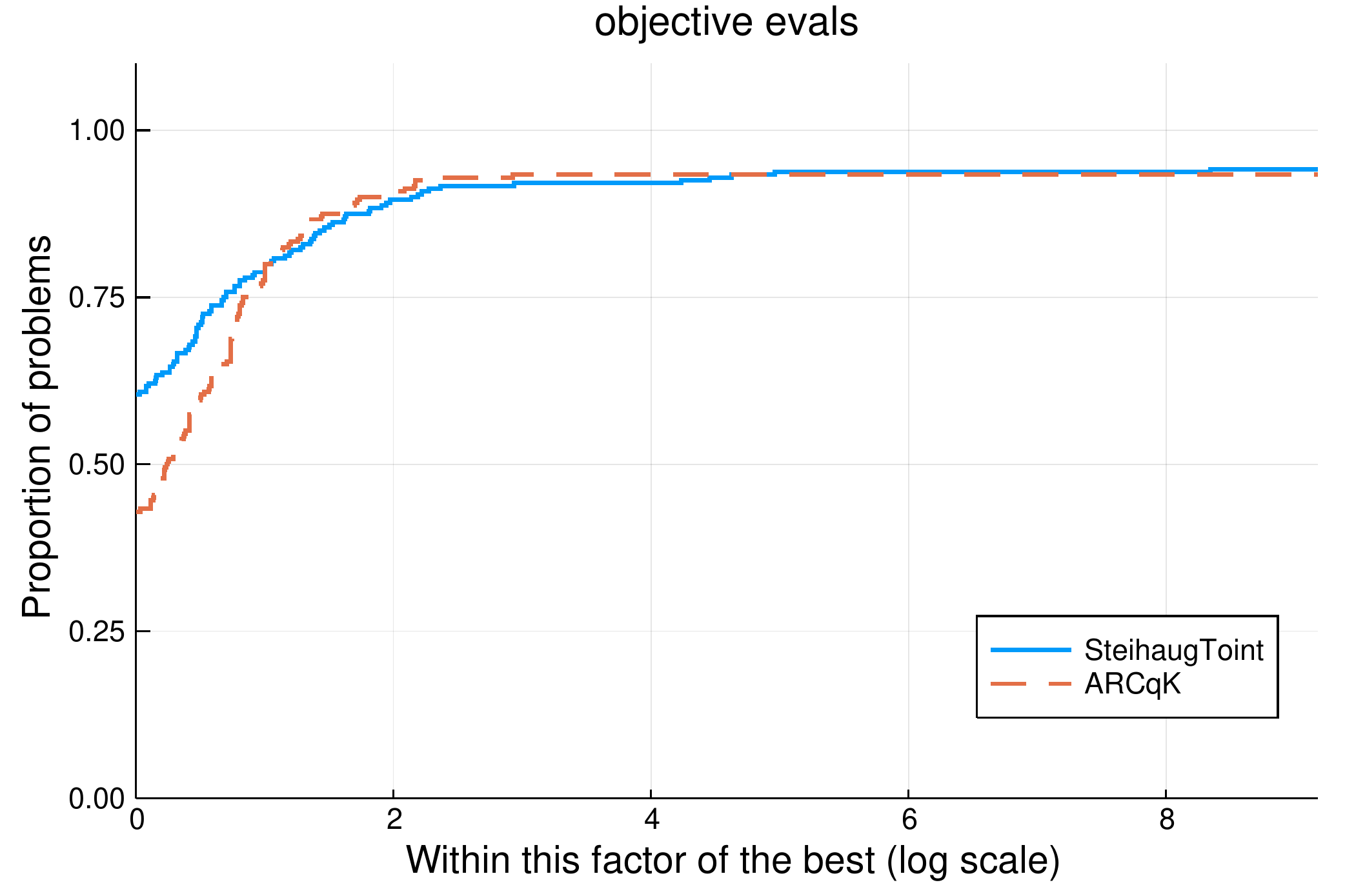}
    \hfill
    \includegraphics[height=.185\textheight]{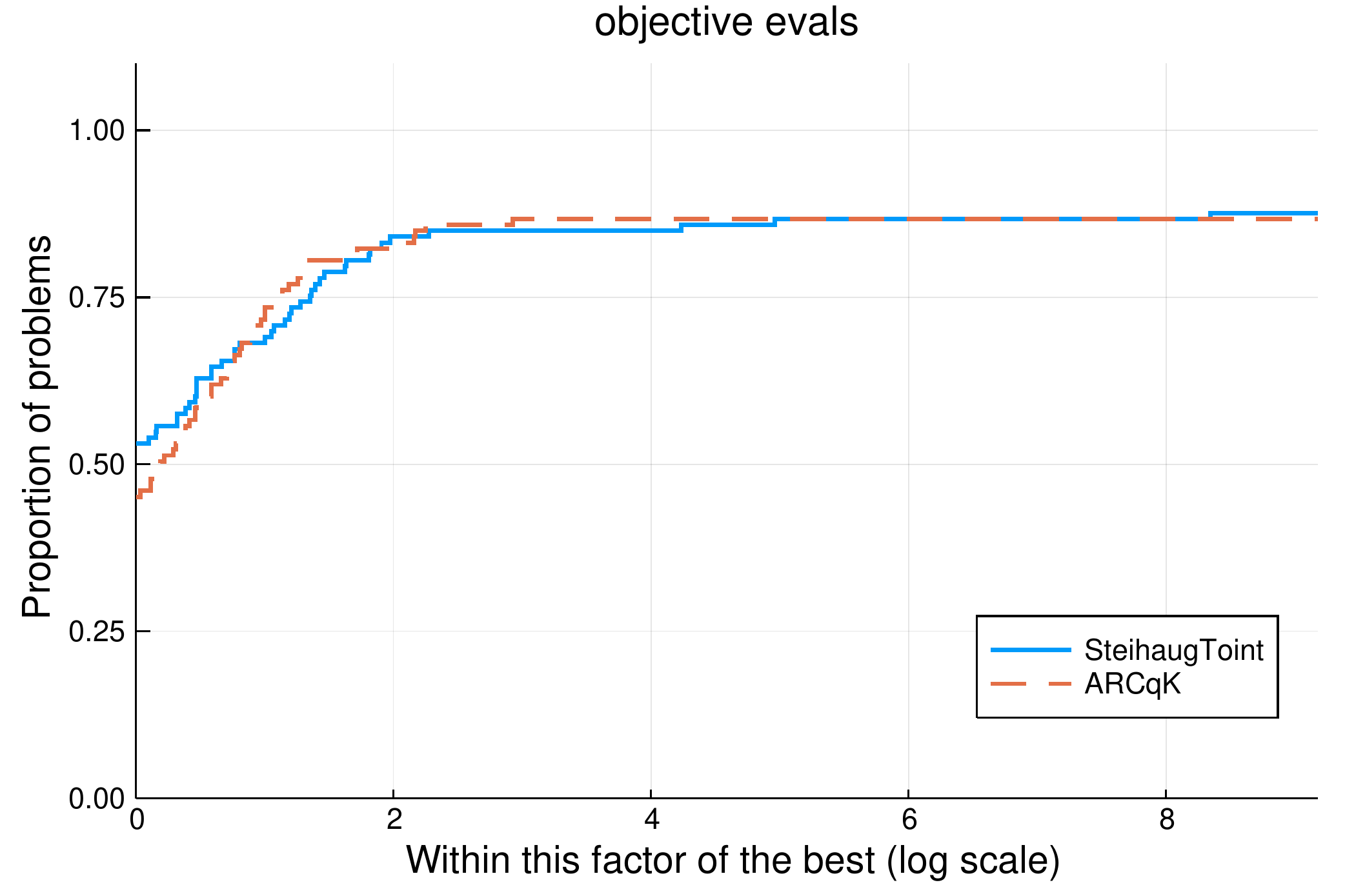}
    \\ 
    \includegraphics[height=.185\textheight]{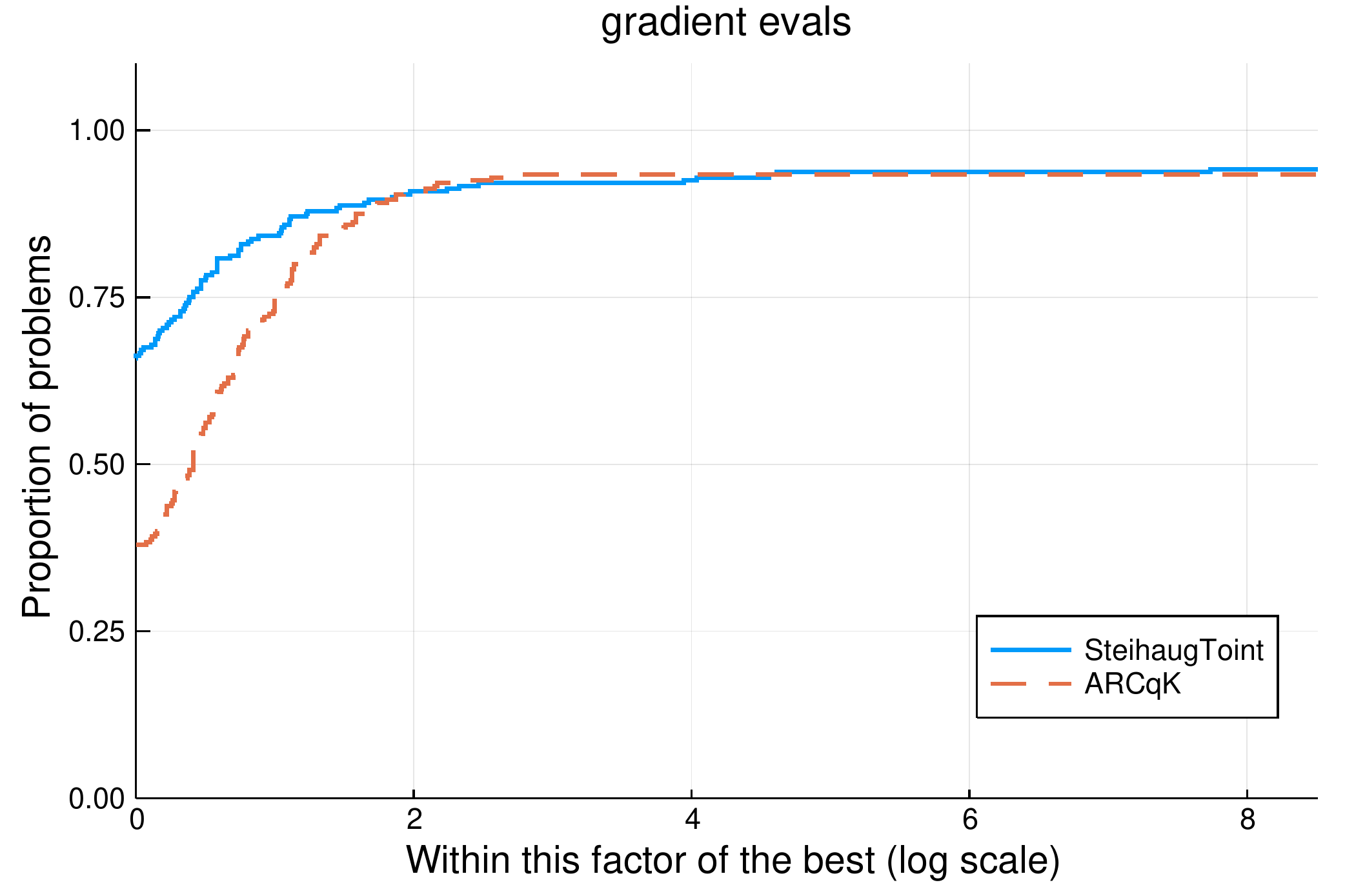}
    \hfill
    \includegraphics[height=.185\textheight]{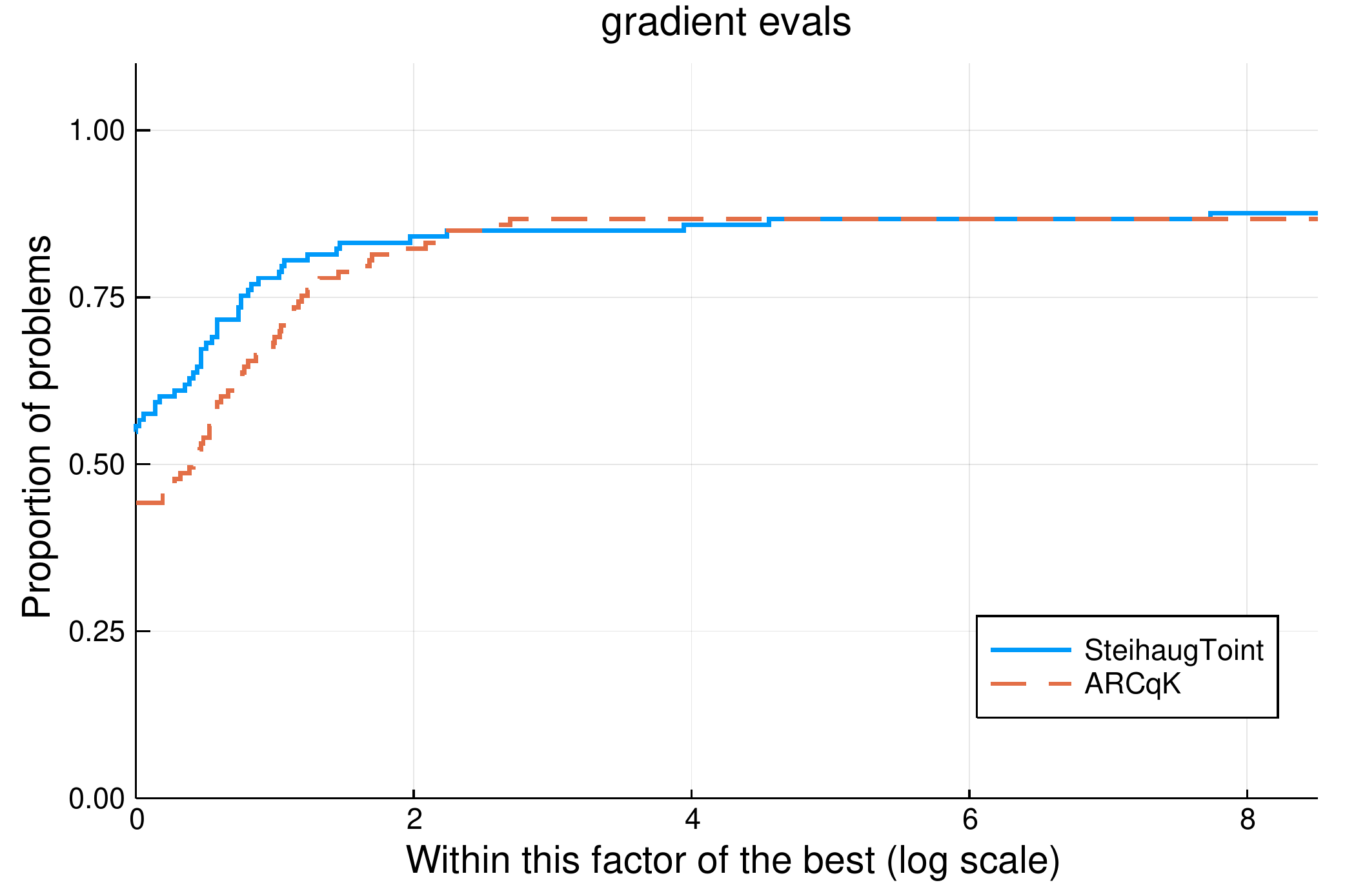}
    \\ 
    \includegraphics[height=.185\textheight]{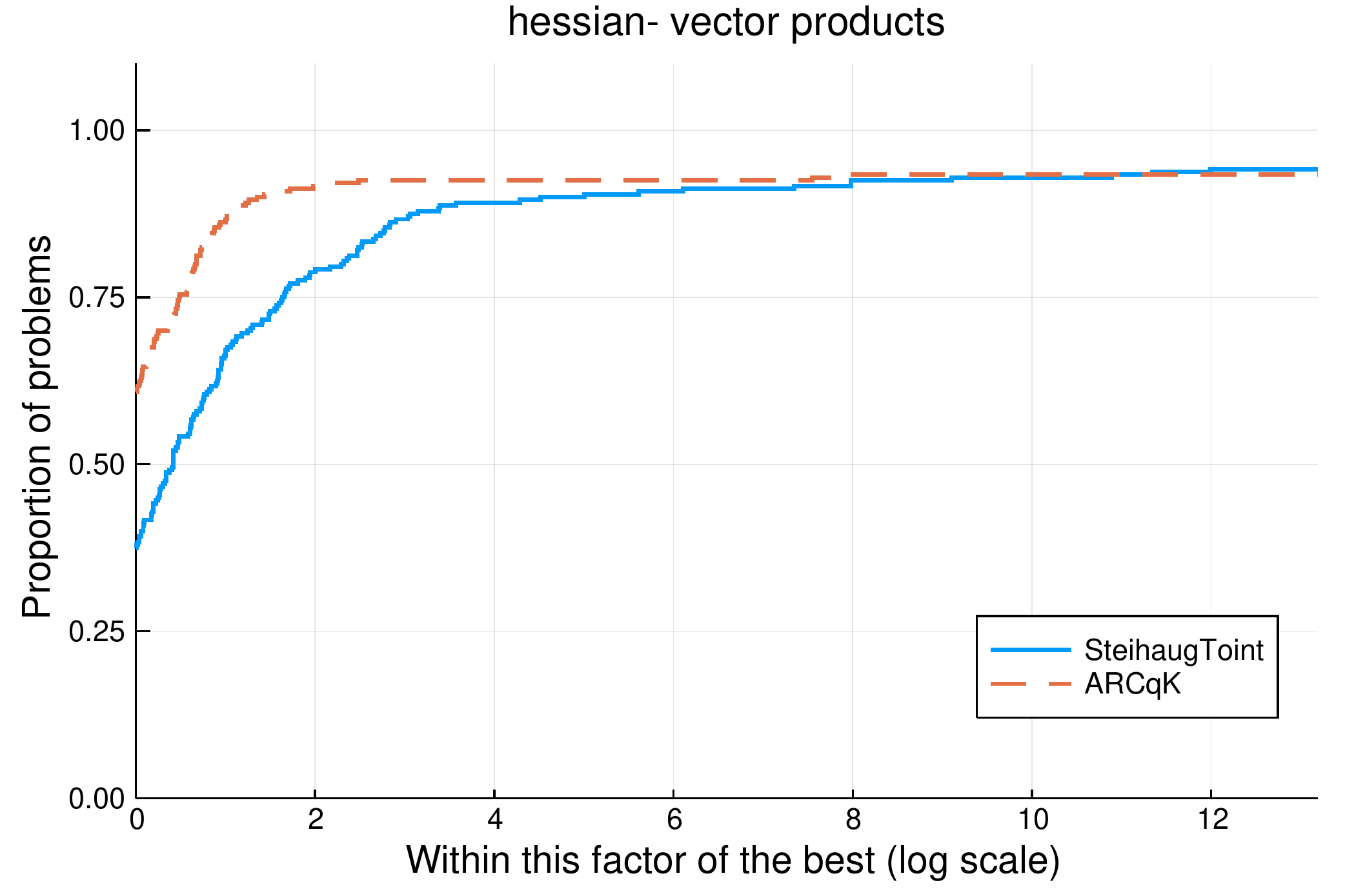}
    \hfill
    \includegraphics[height=.185\textheight]{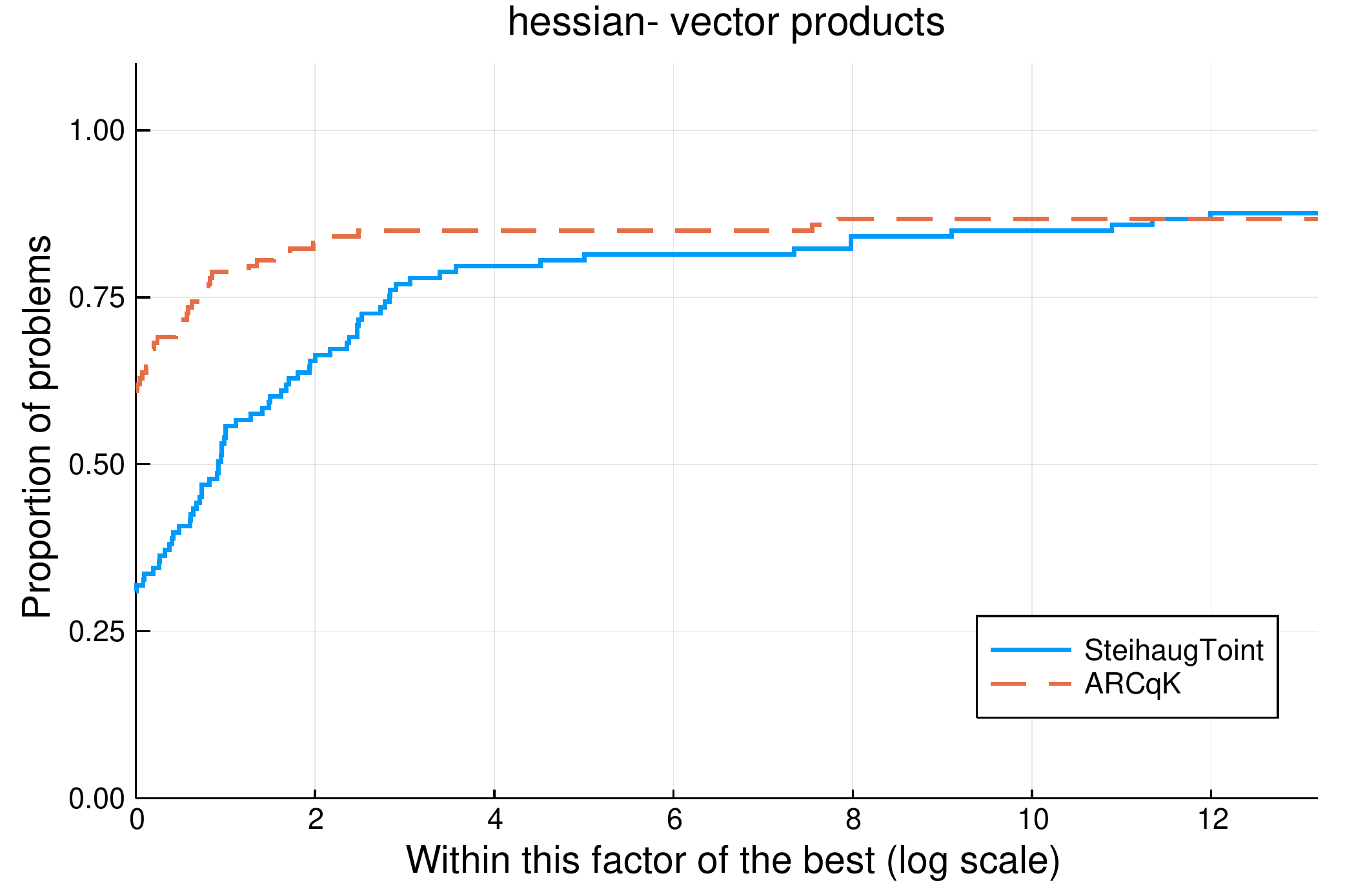}
    \\ 
    \includegraphics[height=.185\textheight]{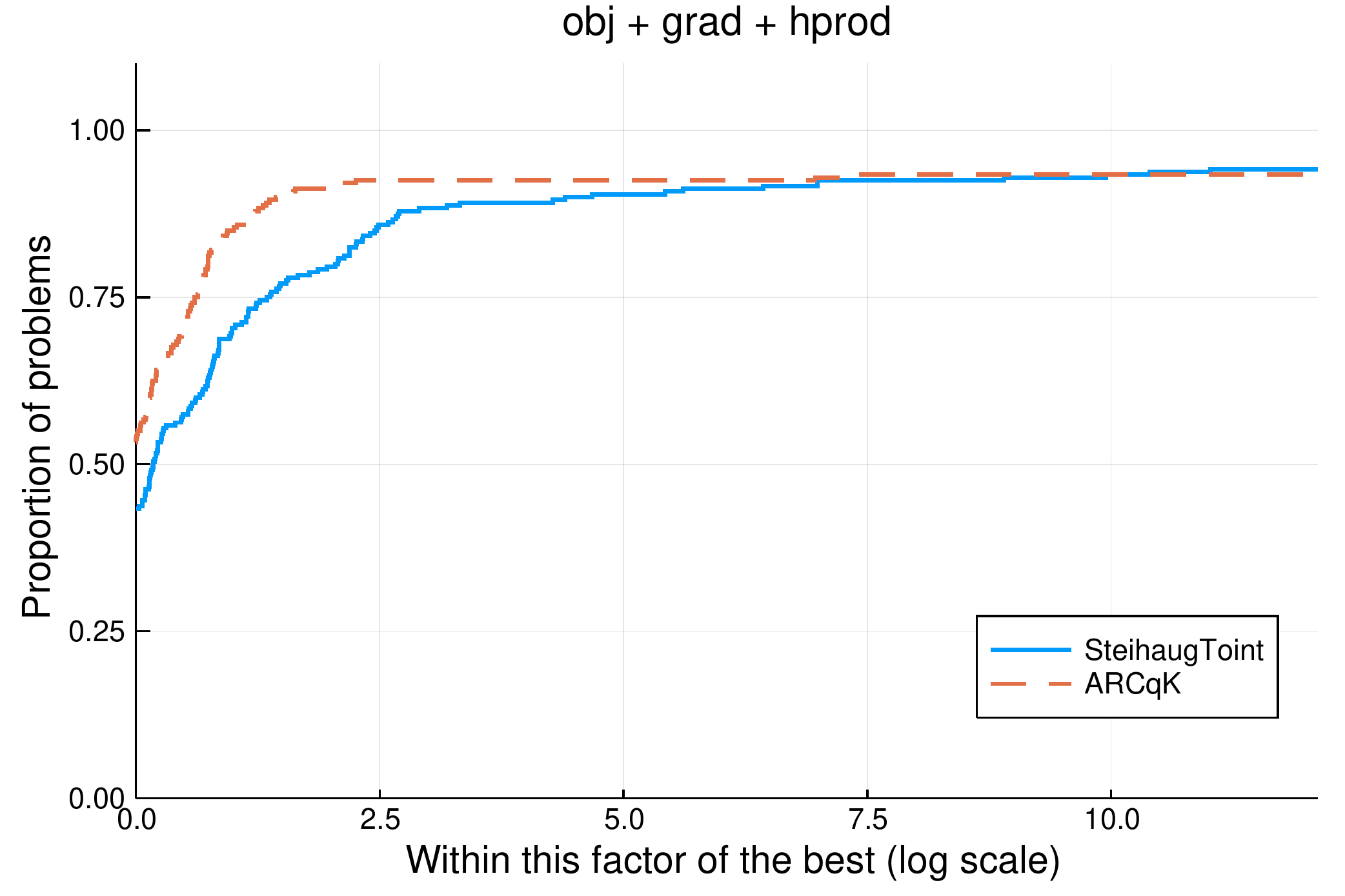}
    \hfill
    \includegraphics[height=.185\textheight]{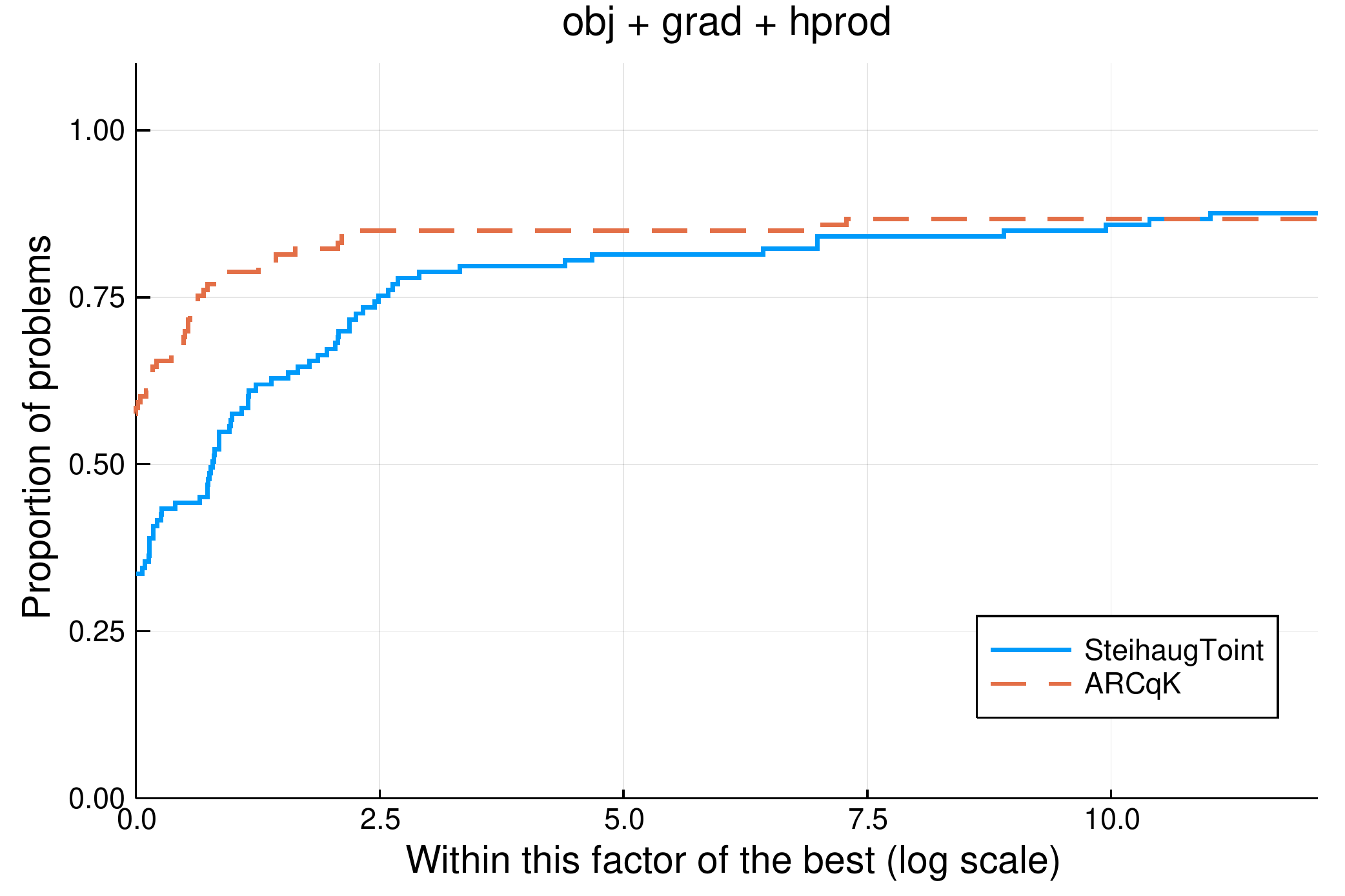}
  \end{center}
  \caption{%
    \label{fig:arc-st}
    Performance profiles comparing \arcqk with a standard truncated conjugate-gradient approach on \(240\) unconstrained \textsf{CUTEst} problems (left) and the \(113\) problems with \(n \geq 100\) (right).
  }
\end{figure}

\section{Extensions}

\subsection{Inexact Hessians}

In many practical situations, the exact Hessian \(\nabla^2 f(x_k)\) is either not available or too costly to evaluate, and we may wish to use an approximation.
Thus, instead of~\eqref{eq:quad-model}, our quadratic model at iteration \(k\) becomes
\[
  q_k(d) = f(x_k) + \nabla f(x_k)^T d + \tfrac{1}{2} d^T B_k d,
\]
where \(B_k = B_k^T\).
We update \Cref{asm:approxmin} as follows.

%
%
%
\begin{assumptionc}{asm:globalmin}%
  \label{asm:approxmin-inexact}
  At line~\ref{line:globalmin} of \Cref{alg:ARCq}, we compute a direction \(d\) satisfying~\eqref{eq:Kr1}--\eqref{eq:Kr3} with \(B_k\) in place of \(\nabla^2 f(x_k)\) in~\eqref{eq:Kr1} and~\eqref{eq:Kr2}.
\end{assumptionc}

With this change, we may verify that the results of \Cref{sec:convergence} can be adapted to accommodate inexact Hessians.
Under \Cref{asm:approxmin-inexact}, \Cref{le:deltaqmK,le:alphaMinK} are unchanged.
We now make the following additional assumption on the quality of the approximation, which was also used with \(\zeta = 1\) by \citet{cartis-gould-toint-2011a,cartis-gould-toint-2011b}.

\begin{assumption}%
  \label{asm:inexact-hessian}
  There exists a constant \(C_H > 0\) such that for all \(k\),
  \[
    \| (\nabla^2 f(x_k) - B_k) d_k \| \leq C_H \|d_k\|^{1 + \zeta},
  \]
  where \(\zeta\) is as in \Cref{asm:relaxedresid}.
\end{assumption}

We now verify that if an inexact step \(d_k\) is computed so that the residual \(r_k\) satisfies \Cref{asm:relaxedresid}, \Cref{le:quadK-relaxedresid} continues to apply.

\begin{lemma}%
  \label{le:quadK-relaxedresid-inexact-hessian}
  Let \Cref{asm:C2,asm:d2f-lipschitz,asm:approxmin-inexact,asm:relaxedresid} be satisfied and assume that \(r_k^T d_k = 0\).
  Then \(\|\nabla f(x_{k+1})\| \le \kappa_g^{-1} (\|d_k\|^2 + \|d_k\|^{1 + \zeta})\) for each successful iteration \(k\), where
  \(\kappa_g :=  2 \, \max(\frac12 L + \beta / \alpha_{\min}, \, C_H + \xi)\).
\end{lemma}

\begin{proof}
  Proceeding as in the proof of \Cref{le:quadK}, we bound \(\|\nabla f(x_{k+1})\|\) above with
  \begin{align*}
    & \left\| \nabla f(x_k) + \int_0^1 \nabla^2 f(x_k+\tau d_k) d_k \, \mathrm{d}\tau \right\|
    \\ = \ &
    \left\|\int_0^1 \nabla^2 f(x_k+\tau d_k) d_k \, \mathrm{d}\tau - (B_k+\lambda_k I) d_k + r_k \right\|
    \\ = \ &
    \left\|\int_0^1 (\nabla^2 f(x_k+\tau d_k) -\nabla^2 f(x_k)) d_k \, \mathrm{d}\tau + (\nabla^2 f(x_k) - B_k) d_k -\lambda_k d_k + r_k \right\|
    \\ \le \ &
    \|d_k\| \, \left\|\int_0^1 L\tau d_k \, \mathrm{d}\tau \right\| + C_H \|d_k\|^{1 + \zeta} + \lambda_k \|d_k\| + \|r_k\|.
  \end{align*}
  Now,~\eqref{eq:Kr3}, \Cref{asm:relaxedresid,asm:inexact-hessian}, and \Cref{le:alphaMinK} yield
  \[
    \|\nabla f(x_{k+1})\| \le
    (\tfrac{1}{2} L + \beta / \alpha_{\min}) \, \|d_k\|^{2} + (C_H + \xi) \, \|d_k\|^{1 + \zeta}.
  \]
  We may now conclude as in the proof of \Cref{le:quadK-relaxedresid}.
\end{proof}

The complexity bound of \Cref{cor:CplxARCqK} continues to apply.
Finally, the local rate given in \Cref{th:SupLinARCqK} also continues to apply with a slight update to the proof.
We state it for completeness.

\begin{theorem}[Superlinear convergence of \arcq]%
  \label{th:SupLinARCqK-inexact}
  Let \Cref{asm:C2,asm:d2f-lipschitz,asm:lbnd,asm:approxmin-inexact,asm:alpha,asm:relaxedresid,asm:2nd-order,asm:inexact-hessian} be satisfied and assume that \(r_k^T d_k = 0\) for all sufficiently large \(k\).
  Then there exists an iteration \(K\) such that for all \(k \ge K\),~\eqref{eq:arcq-super} holds.
\end{theorem}

\begin{proof}
  We proceed as in the proof of \Cref{th:SupLinARCqK} with the difference that we bound \(\|d_k\|\) as follows.
  \Cref{asm:approxmin-inexact,asm:inexact-hessian} allow us to write
  \begin{align*}
    \|r_k - \nabla f(x_k)\| & \leq
    \| (\nabla^2 f(x_k) + \lambda_k I) d_k \| + \|(\nabla^2 f(x_k) - B_k) d_k\|
    \\ & \leq \lambda_{\min}^k \|d_k\| + C_H \|d_k\|^{1 + \zeta}
    \\ & \leq \tfrac{1}{2} \lambda_{\min}^* \|d_k\| + C_H \|d_k\|^{1 + \zeta}.
  \end{align*}
  Because \(\{d_k\} \to 0\), we can choose \(K\) large enough that \(C_H \|d_k\|^{1 + \zeta} \leq \tfrac{1}{2} \lambda_{\min}^* \|d_k\|\) for all \(k \geq K\).
  \Cref{asm:relaxedresid} now yields
  \[
    \|d_k\| \leq
    \frac{\|r_k\| + \|\nabla f(x_k)\|}{\lambda_{\min}^*},
  \]
  and the proof concludes as in \Cref{th:SupLinARCqK}.
\end{proof}

\subsection{Nonlinear Least-Squares Problems}

A prime application of the extension to inexact Hessians occurs in the context of nonlinear least-squares problems, where
\[
  f(x) = \tfrac{1}{2} \|F(x)\|^2, \qquad
  F(x) =
  \begin{bmatrix}
    f_1(x) \\ \vdots \\ f_m(x)
  \end{bmatrix},
\]
and each \(f_i : \R^n \to \R\), \(i = 1, \dots, m\).
We begin with the following assumption.

\begin{assumption}%
  \label{asm:lsq-C2}
  Each \(f_i\), \(i = 1, \dots, m\), is twice continuously differentiable.
\end{assumption}

Evidently, \Cref{asm:lsq-C2} implies \Cref{asm:C2}.

For reference, we note that
\[
  \nabla f(x) = J(x)^T F(x)
  \quad \text{where} \quad
  J(x) :=
  \begin{bmatrix}
    \nabla f_1(x)^T
    \\ \vdots \\
    \nabla f_m(x)^T
  \end{bmatrix},
\]
and
\[
  \nabla^2 f(x) = J(x)^T J(x) + \sum_{i=1}^m f_i(x) \nabla^2 f_i(x).
\]

We follow \citet{cartis-gould-toint-2013} and make the following additional assumption.

\begin{assumption}%
  \label{asm:lsq-lipschitz}
  Each \(f_i\), \(\nabla f_i\) and \(\nabla^2 f_i\), \(i = 1, \dots, m\) is globally Lipschitz continuous.
\end{assumption}

Under \Cref{asm:lsq-lipschitz}, it is possible to show that \Cref{asm:d2f-lipschitz} is satisfied.

\Cref{asm:lbnd} is not necessary in the least-squares context as we may simply set \(f_{\text{low}} = 0\).

In practice, it is common to use the Gauss-Newton approximation \(J(x)^T J(x)\) instead of \(\nabla^2 f(x)\), which is justified if the residuals \(f_i\) are all either nearly zero or nearly linear around a local minimizer.
Thus we consider the following variant of \Cref{asm:inexact-hessian}.

\begin{assumption}%
  \label{asm:gauss-newton}
  There exists a constant \(C_H > 0\) such that for all \(k\),
  \[
    \|\sum_{i=1}^m f_i(x_k) \nabla^2 f_i(x_k) d_k\| \leq C_H \, \|d_k\|^{1 + \zeta},
  \]
  where \(\zeta\) is as in \Cref{asm:relaxedresid}.
\end{assumption}

Under the above assumption, the convergence and complexity properties of the previous sections apply to the nonlinear least-squares problem.
Obviously, should \Cref{asm:gauss-newton} be too restrictive, a different Hessian approximation satisfying \Cref{asm:inexact-hessian} can be used.
In the sequel, we focus on the Gauss-Newton approximation because of its ubiquity.

The solution of the shifted system
\begin{equation}%
  \label{eq:normal-eqns}
  (J(x_k)^T J(x_k) + \lambda_k I) d_k = -J(x_k)^T F(x_k),
\end{equation}
known as the \emph{regularized normal equations}, can be interpreted as the solution of the regularized linear least-squares problem
\begin{equation}
  \label{eq:reg-lsq}
  \minimize{d} \ \tfrac{1}{2} \left\|
    \begin{bmatrix}
         J(x_k)
      \\ \sqrt{\lambda_k} I
    \end{bmatrix}
    d +
    \begin{bmatrix}
      F(x_k) \\ 0
    \end{bmatrix}
  \right\|^2.
\end{equation}
It is well-known that applying CG (or CG-Lanczos) directly to~\eqref{eq:normal-eqns} is prone to accumulation of rounding errors because \(J(x_k)^T J(x_k)\) can be substantially more ill conditioned than \(J(x_k)\).
Several iterative methods exist for~\eqref{eq:reg-lsq}, but a straightforward approach consists in modifying \Cref{alg:slcg} to accommodate the structure of~\eqref{eq:normal-eqns} and decouple the product with \(J(x_k)\) from that with \(J(x_k)^T\).
Such a modification of CG is known as CGLS and several implementations appear in the literature, the first of which was stated by \citet{hestenes-stiefel-1952}.
As it turns out, \citet{frommer-maass-1999} were also interested in~\eqref{eq:normal-eqns}.
We use their formulation, stated as \Cref{alg:slcgls} for the generic regularized normal equations
\begin{equation}%
  \label{eq:generic-normal-eqns}
  (A^T A + \lambda I) x = A^T b,
\end{equation}
where \(A\) can be rectangular.

\begin{algorithm}[htbp]
  \caption{Lanczos-CGLS with shifts for~\eqref{eq:generic-normal-eqns}}%
  \label{alg:slcgls}
  \begin{algorithmic}[1]
    \State Set \(\bm{x}_0 = 0\), \(u_0 = b\), \(\beta_0 v_0 = A^T u_0\), \(u_0 = u_0 / \beta_0\), \(\bm{p}_0 = \beta_0 v_0\) \Comment{\(\beta_0\) such that \(\|v_0\| = 1\)}
    \State set \(u_{-1} = 0\), \(\bm{\sigma}_0 = \beta_0\), \(\bm{\omega}_{-1} = 0\),
      \(\bm{\gamma}_{-1} = 1\), \(\bm{\pi}_0 = \beta_0^2\) \Comment{\(\bm{\pi}_0 = \|\bm{p}_0\|^2\)}
    \For{\(j \gets 0, 1, 2, \dots\)}
      \State\label{alg:slcgls-uj-tilde}%
        \(\tilde{u}_j = A v_j\)
        \Comment{Lanczos part of the iteration}
      \State \(\delta_j = \tilde{u}_j^T \tilde{u}_j\)
        \Comment{\(\delta_j = v_j^T A^T A v_j\)}
      \State \(u_{j+1} = \tilde{u}_j - \delta_j u_j - \beta_j u_{j-1}\)
      \State \(\beta_{j+1} v_{j+1} = A^T u_{j+1}\) \Comment{\(\beta_{j+1}\) such that \(\|v_{j+1}\| = 1\)}
      \State \(u_{j+1} = u_{j+1} / \beta_{j+1}\)
      \State \(\bm{\delta}_j = \delta_j + \bm{\lambda}\)
        \Comment{CGLS part of the iteration in block form}
      \State \(\bm{\gamma}_j = 1 / (\bm{\delta}_j - \bm{\omega}_{j-1} / \bm{\gamma}_{j-1})\)
      \State \(\bm{\omega}_j = (\beta_{j+1} \bm{\gamma}_j)^2\)
      \State \(\bm{\sigma}_{j+1} = -\beta_{j+1} \bm{\gamma}_j \bm{\sigma}_j\) \Comment{\(\bm{\sigma}_{j+1} = \|\bm{r}_{j+1}\|\)}
      \State \(\bm{x}_{j+1} = \bm{x}_j + \bm{\gamma}_j \bm{p}_j\)
      \State \(\bm{p}_{j+1} = \bm{\sigma}_{j+1} v_{j+1} + \bm{\omega}_j \bm{p}_j\)
      \State 
        \(\bm{\pi}_{j+1} = \bm{\sigma}_{j+1}^2 + \bm{\omega}_j^2  \bm{\pi}_j\) \Comment{\(\bm{\pi}_{j+1} = \|\bm{p}_{j+1}\|^2\)}
    \EndFor
  \end{algorithmic}
\end{algorithm}

If \(A\) has size \(m\)\(\times\)\(n\), each \(v_j \in \R^n\) and we see that \Cref{alg:slcgls} also updates auxiliary vectors \(u_j \in \R^m\).
Typical least-squares problems are over-determined, i.e., \(m > n\), although the algorithm remains well defined for \(m \leq n\).
The vector \(\tilde{u}_j\) at line~\ref{alg:slcgls-uj-tilde} can share storage with \(u_{j+1}\).

Note that the residual vector whose norm is updated as \(\bm{\sigma}_{j+1}\) now represents the optimality residual \(\bm{r}_j := A^T (b - A \bm{x}_j)\).
\Cref{alg:slcgls} can be further improved by monitoring the least-squares residual \(\bm{s}_j := b - A \bm{x}_j\).
The latter would be initialized as \(\bm{s}_0 = u_0 = b\) and simply updated as \(\bm{s}_{j+1} = \bm{\sigma}_j u_j\) after \(u_j\) has been normalized.
Contrary to the \(v_j\), the \(u_j\) are not orthogonal so that \(\|\bm{s}_j\|\) must be explicitly computed if it is of interest.

A final difference with \Cref{alg:slcg} is that \Cref{alg:slcgls} need not be concerned with detection of negative curvature.

\section*{Discussion}%
\label{sec:Ccl}

We introduced a new scalable factorization-free implementation of the \arcq variant of the adaptive regularization by cubics.
An inexact search direction is computed based on the simultaneous solution of a set of shifted linear systems.
The resulting method, named \arcqk, enjoys a complexity bound of \(O(\epsilon^{-3/(1 + \zeta)})\) where \(\zeta \in (0, 1]\) guides the accuracy of the system solves, and converges superlinearly.
Our implementation is competitive with the method of~\citet{steihaug-1983} and~\citet{toint-1981} on the unconstrained \textsf{CUTEst} problems in terms of time, and number of objective and gradient evaluations.
\arcqk has a noticeable advantage in terms of number of Hessian-vector products.

We also discussed a variant of \arcqk for nonlinear least-squares problems.
Our complexity bound continues to apply if the Gauss-Newton Hessian approximation is accurate along the step.

Certainly, much work remains to be done to develop a reliable and efficient version of \Cref{alg:ARCqK}, but the numerical results of this section speak in favor of our approach.
Among possible improvements, we note that our current implementation solves all the shifted systems up front.
Although this could be performed in parallel, there may be virtue to only solving a system when it is needed.
The CG iterations could be put on hold when a shifted system has been solved to the desired accuracy, and resumed only if it is determined that a system corresponding to a larger value of the shift should be solved.

In practice, preconditioning CG is important if one is to develop a robust implementation.
Of course, preconditioning is strongly application dependent, but generic (e.g., diagonal or banded) preconditioners can prove generally useful.

The numerical stability of the so-called \emph{multishift} CGLS has been studied for some time in the literature and several variants have been proposed, each with its own numerical properties.
In practice, implementations tend to depend on the condition number of \(A^T A\) rather than that of \(A\) even though it is not the case when applying CGLS to solve a single shifted system.
\Citet{vandeneshof-sleijpen-2004} provide a review, numerous references, and an implementation of the multishift CGLS with favorable stability properties, although it is based on the standard CG rather than the Lanczos variant of CG\@.
In this paper, we used the implementation of \citet{frommer-maass-1999} and its stability properties should be properly evaluated.

A possible improvement over CGLS is to use a variant of LSQR \citep{paige-saunders-1982} adapted to solve regularized linear least-squares problems with multiple values of the regularization parameter.
However a disadvantage of CG and CGLS is that they do not reduce the residual norm \(\|r_k\|\) monotonically.
In the inexact-Newton context of the present research, an iterative method that reduces the residual norm monotonically could be preferable.
Such methods include the method of conjugate residuals, also introduced by \citet{hestenes-stiefel-1952} for symmetric positive-definite systems, or its MINRES variant of \citet{paige-saunders-1975} for general symmetric systems, and LSMR of \citet{fong-saunders-2011} for least squares.
Variants of some of those methods to accommodate multiple shifts are described by \citet{glassner-gusken-lippert-ritzenhofer-schilling-frommer-1996} and \citet{jegerlehner-1996}.

On the optimization side, \citet{gould-porcelli-toint-2012} suggest more sophisticated parameter update strategies for least-squares problems that could prove beneficial to \arcqk.

%% file: arcqk_full_results.tex
\begin{longtable}[c]{lrrrrrrrrr}
\hline 
name & nvar & \(f(x)\) & \(\|g\|\) & iter & \#\(f\) & \#\(g\) & \#\(Hv\) & \(t\) & stat \\
\hline 
\endfirsthead
\multicolumn{10}{l}
{{\bfseries \tablename\ \thetable{} --- continued from previous page}} \\
\hline 
name & nvar & \(f(x)\) & \(\|g\|\) & iter & \#\(f\) & \#\(g\) & \#\(Hv\) & \(t\) & stat \\
\hline 
\endhead
\hline 
\multicolumn{10}{r}{{\bfseries Continued on next page}} \\
\hline 
\endfoot
\hline 
\endlastfoot
akiva & \(2\) & \( 6.2\)e\(+00\) & \( 2.2\)e\(-04\) & \(7\) & \(8\) & \(8\) & \(16\) & \( 4.2\)e\(-03\) &  \\
allinitu & \(4\) & \( 5.7\)e\(+00\) & \( 1.8\)e\(-06\) & \(18\) & \(19\) & \(15\) & \(39\) & \( 4.5\)e\(-03\) &  \\
arglina & \(200\) & \( 2.0\)e\(+02\) & \( 2.2\)e\(-11\) & \(4\) & \(5\) & \(5\) & \(8\) & \( 1.2\)e\(-02\) &  \\
arglinb & \(200\) & \( 1.0\)e\(+02\) & \( 1.3\)e\(+01\) & \(1\) & \(2\) & \(2\) & \(2\) & \( 4.5\)e\(-03\) &  \\
arglinc & \(200\) & \( 1.0\)e\(+02\) & \( 1.2\)e\(+01\) & \(1\) & \(2\) & \(2\) & \(2\) & \( 4.5\)e\(-03\) &  \\
argtrigls & \(200\) & \( 5.6\)e\(-09\) & \( 2.2\)e\(-03\) & \(14\) & \(15\) & \(15\) & \(516\) & \( 9.0\)e\(-01\) &  \\
arwhead & \(5,000\) & \( 1.1\)e\(-12\) & \( 5.8\)e\(-06\) & \(5\) & \(6\) & \(6\) & \(11\) & \( 3.5\)e\(-02\) &  \\
ba-l16ls & \(66,462\) & \( 9.3\)e\(+06\) & \( 2.5\)e\(+05\) & \(36\) & \(37\) & \(27\) & \(12,095\) & \( 1.3\)e\(+04\) & e \\
\hline 
\end{longtable}

%% file: st_full_results.tex
\begin{longtable}[c]{lrrrrrrrrr}
\hline 
name & nvar & \(f(x)\) & \(\|g\|\) & iter & \#\(f\) & \#\(g\) & \#\(Hv\) & \(t\) & stat \\
\hline 
\endfirsthead
\multicolumn{10}{l}
{{\bfseries \tablename\ \thetable{} --- continued from previous page}} \\
\hline 
name & nvar & \(f(x)\) & \(\|g\|\) & iter & \#\(f\) & \#\(g\) & \#\(Hv\) & \(t\) & stat \\
\hline 
\endhead
\hline 
\multicolumn{10}{r}{{\bfseries Continued on next page}} \\
\hline 
\endfoot
\hline 
\endlastfoot
akiva & \(2\) & \( 6.2\)e\(+00\) & \( 1.3\)e\(-06\) & \(6\) & \(7\) & \(7\) & \(17\) & \( 3.1\)e\(-03\) &  \\
allinitu & \(4\) & \( 5.7\)e\(+00\) & \( 9.4\)e\(-07\) & \(6\) & \(7\) & \(6\) & \(20\) & \( 1.2\)e\(-03\) &  \\
arglina & \(200\) & \( 2.0\)e\(+02\) & \( 1.8\)e\(-13\) & \(2\) & \(3\) & \(3\) & \(4\) & \( 6.8\)e\(-03\) &  \\
arglinb & \(200\) & \( 1.0\)e\(+02\) & \( 8.4\)e\(-02\) & \(2\) & \(3\) & \(3\) & \(4\) & \( 6.7\)e\(-03\) &  \\
arglinc & \(200\) & \( 1.0\)e\(+02\) & \( 4.6\)e\(-02\) & \(2\) & \(3\) & \(3\) & \(4\) & \( 6.7\)e\(-03\) &  \\
argtrigls & \(200\) & \( 4.0\)e\(-13\) & \( 2.4\)e\(-05\) & \(7\) & \(8\) & \(6\) & \(846\) & \( 1.4\)e\(+00\) &  \\
arwhead & \(5,000\) & \( 3.3\)e\(-12\) & \( 3.2\)e\(-04\) & \(5\) & \(6\) & \(6\) & \(14\) & \( 2.8\)e\(-02\) &  \\
ba-l16ls & \(66,462\) & \( 7.9\)e\(+06\) & \( 3.4\)e\(+07\) & \(24\) & \(25\) & \(17\) & \(3,971\) & \( 4.1\)e\(+03\) & e \\
\hline 
\end{longtable}